\documentclass[11pt,reqno]{amsart}
\usepackage[T1]{fontenc}
\usepackage{lmodern,amsmath,amssymb,amsthm,mathtools,mathrsfs}
\usepackage{booktabs,longtable,array}
\usepackage[a4paper,margin=27mm]{geometry}
\usepackage{microtype}
\usepackage[hidelinks,hypertexnames=false]{hyperref}
\usepackage{enumitem}
\allowdisplaybreaks[2]
\numberwithin{equation}{section}
\setlist[enumerate]{leftmargin=2.2em,label=\textup{(\roman*)}}
\newtheorem{theorem}{Theorem}[section]
\newtheorem{proposition}[theorem]{Proposition}
\newtheorem{lemma}[theorem]{Lemma}
\newtheorem{corollary}[theorem]{Corollary}
\theoremstyle{definition}
\theoremstyle{remark}
\newcommand{\C}{\mathbb C}\newcommand{\R}{\mathbb R}\newcommand{\Z}{\mathbb Z}
\newcommand{\Q}{\mathbb Q}\newcommand{\GL}{\operatorname{GL}}\newcommand{\SL}{\operatorname{SL}}
\newcommand{\Sp}{\operatorname{Sp}}\newcommand{\SO}{\operatorname{SO}}\newcommand{\Or}{\operatorname{O}}
\newcommand{\Spin}{\operatorname{Spin}}\newcommand{\Lie}{\operatorname{Lie}}\newcommand{\Ad}{\operatorname{Ad}}
\newcommand{\ad}{\operatorname{ad}}\newcommand{\tr}{\operatorname{tr}}
\newcommand{\Sym}{\operatorname{Sym}}\newcommand{\Hom}{\operatorname{Hom}}\newcommand{\End}{\operatorname{End}}
\newcommand{\supp}{\operatorname{supp}}\newcommand{\diag}{\operatorname{diag}}
\newcommand{\rk}{\operatorname{rank}}\newcommand{\WF}{\operatorname{WF}}
\newcommand{\fg}{\mathfrak g}\newcommand{\fh}{\mathfrak h}\newcommand{\fp}{\mathfrak p}\newcommand{\fz}{\mathfrak z}
\newcommand{\cN}{\mathcal N}\newcommand{\cS}{\mathcal S}\newcommand{\DD}{\mathcal D}
\newcommand{\cF}{\mathcal F}\newcommand{\F}{\mathcal F_B}
\newcommand{\absC}[1]{\lvert#1\rvert_{\C}}
\newcommand{\sslash}{/\!\!/}

\newcommand{\fl}{\mathfrak l}\newcommand{\fk}{\mathfrak k}
\newcommand{\Palg}{\mathscr P}\newcommand{\Ialg}{\mathscr I}\newcommand{\Ralg}{\mathscr R}
\newcommand{\U}{\operatorname{U}}\newcommand{\USp}{\operatorname{USp}}
\newcommand{\id}{\operatorname{id}}
\newcommand{\ii}{\mathbf{i}}

\newcommand{\cC}{\mathcal C}
\DeclareMathOperator{\Aut}{Aut}
\DeclareMathOperator{\Int}{Int}
\DeclareMathOperator{\im}{im}

\title[Regularity of complex symmetric pairs]{Regularity of complex symmetric pairs}
\author[Y.F.~Li]{Yufeng Li}
\address{School of Mathematical Sciences, Peking University, No. 5 Yiheyuan Road, Beijing 100871, P.R. China}
\email{2401110016@stu.pku.edu.cn}

\author[J.Y.~Xiao]{Junyan Xiao}
\address{School of Mathematical Sciences, Peking University, No. 5 Yiheyuan Road, Beijing 100871, P.R. China}
\email{jyxiao25@stu.pku.edu.cn}
 
\author[J.~Yu]{Jun Yu}
\address{School of Mathematical Sciences and Beijing International Center for Mathematical Research, Peking University, 
No. 5 Yiheyuan Road, Beijing 100871, P.R. China}
\email{junyu@bicmr.pku.edu.cn}
\date{October 9, 2026}
\hypersetup{pdfauthor={Yufeng Li, Junyan Xiao, Jun Yu},pdftitle={Regularity of complex symmetric pairs},
pdfsubject={The complex regularity conjecture with arbitrary component groups}}
\begin{document}
\raggedbottom
\begin{abstract}
We prove that every complex symmetric pair whose ambient group has
reductive identity component is regular. This proves the
Aizenbud--Gourevitch regularity conjecture over the complex numbers.
No connectedness assumption
is imposed on the group or its fixed subgroup. The proof reduces the
problem to indecomposable symmetric pairs of Lie algebras. A Gaussian argument
handles the classical matrix representations by means of the cotangent
moment map. For EIII and EVII, homogeneity and transverse slices control
distributions supported, together with their Fourier transforms, on the
nilpotent cone. It follows by generalized Harish--Chandra descent that
every good complex symmetric pair has the Gelfand--Kazhdan property.
In particular, a complex symmetric pair with connected homogeneous
space has the unitary Gelfand property, proving van Dijk's conjecture
for connected complex symmetric pairs.
\end{abstract}
\maketitle
\setcounter{tocdepth}{1}\tableofcontents
\section{Introduction}\label{sec:intro}
Let $G$ be a complex algebraic group with reductive identity component,
let $\theta$ be an algebraic involution of $G$, and let $H=G^\theta$ be
its full fixed subgroup. Let $\fg$ be the Lie algebra of $G$, and write
\[
 \fg=\fh\oplus\fp,\qquad
 \fh=\fg^\theta,\quad \fp=\fg^{-\theta},
\]
for the eigenspace decomposition of the differential of $\theta$.
The associated anti-involution is
\[
 \sigma(g)=\theta(g^{-1})\qquad(g\in G).
\]
The pair $(\fg,\fh)$ is the symmetric pair of Lie algebras associated
with $(G,\theta)$. Our main result is the following.

\begin{theorem}\label{thm:main}
Let $(G,G^\theta,\theta)$ be a complex symmetric pair, and let $G^0$
be the reductive identity component of $G$. Then the pair is regular.
\end{theorem}

This theorem settles the Aizenbud--Gourevitch regularity conjecture
\cite[Section~7.5, Conjecture~4]{AG} over the complex numbers. It allows
arbitrary component groups of $G$ and $H$; in particular, the homogeneous
space $G/H$ may be disconnected. Regularity
is the local distributional condition introduced in
\cite[Definitions 7.4.1--7.4.2]{AG}. It controls the contribution of the
nilpotent cone in generalized Harish--Chandra descent. We recall its
definition before explaining the proof.

For a reductive linear action of a complex algebraic group $M$ on a
complex vector space $X$, put
\[
 Q_M(X)=X/X^M,\qquad
 \Gamma_M(X)=\pi^{-1}(\pi(0)),\qquad
 R_M(X)=Q_M(X)\setminus\Gamma_M(X),
\]
where $\pi:Q_M(X)\to Q_M(X)\sslash M$ is the invariant quotient.
The quotient $Q_M(X)$ can be identified with the sum of the nontrivial
isotypic summands of $X$. For a real Nash manifold $Y$, write $\cS(Y)$
for its Schwartz space and $\cS^*(Y)$ for the continuous dual;
we use the Nash Schwartz spaces of \cite{AGSchwartz}. Complex
varieties and groups are regarded as real Nash spaces in this notation.

A linear automorphism $u$ of $X$ is $M$-admissible if it normalizes the
effective $M$-action, its square belongs to that action, and it preserves
every closed $M$-orbit. An element $g\in G$ is admissible for the
symmetric pair if
\[
 \Int(g)\circ\theta=\theta\circ\Int(g),
 \qquad\text{equivalently}\qquad g^{-1}\theta(g)\in Z(G),
\]
and $u=\Ad(g)|_\fp$ is $H$-admissible.
Here $\Int(g)$ is the automorphism of the algebraic group $G$ defined by
\[
 \Int(g)(x)=gxg^{-1}\qquad\text{for every }x\in G,
\]
and $\Ad(g)$ denotes its differential. This distinction matters when $G$
is disconnected. The pair is regular if, for every
such $g$,
\begin{equation}\label{eq:reg}
 \cS^*(R_H(\fp))^H\subseteq\cS^*(R_H(\fp))^u
 \quad\Longrightarrow\quad
 \cS^*(Q_H(\fp))^H\subseteq\cS^*(Q_H(\fp))^u.
\end{equation}
These definitions agree with \cite[Definitions 6.0.1 and
7.4.1--7.4.2]{AG}. A symmetric pair is called connected when $G/H$
is connected, as in \cite[Definition 7.1.1]{AG}.

Theorem~\ref{thm:main} has the following consequence for invariant
distributions and representations. Recall that a symmetric pair is
\emph{good} if every closed $H\times H$-orbit in $G$ is preserved by
$\sigma$.
\begin{corollary}\label{cor:GK}
Let $(G,H,\theta)$ be a good complex symmetric pair.
Then it is a Gelfand--Kazhdan pair:
\begin{equation}\label{eq:GK}
 \cS^*(G)^{H\times H}\subseteq\cS^*(G)^\sigma.
\end{equation}
For every irreducible Casselman--Wallach representation $\pi$ of $G$,
\[
 \dim\Hom_H(\pi,\C)\,
 \dim\Hom_H(\pi^\vee,\C)\leq1.
\]
For every irreducible unitary representation $U$ of $G$,
\[
 \dim\Hom_H(U^\infty,\C)\leq1.
\]
In particular, all these conclusions hold when $G/H$ is connected.
\end{corollary}

Corollary~\ref{cor:GK} proves van Dijk's conjecture for connected complex
symmetric pairs. The three conditions are formulated together in
\cite[Definition 2.2.1]{AGS}; see also \cite[Definition 8.1.2]{AG}.
In this terminology, GP1 is the bound
$\dim\Hom_H(\pi,\C)\leq1$ for every irreducible admissible
representation; GP2 is the product bound involving $\pi$ and
$\pi^\vee$; and GP3 is the single bound for smooth vectors in
irreducible unitary representations. In the Archimedean setting,
the admissible representations in GP1 and GP2 are understood in the
Casselman--Wallach category. The corollary establishes GP2 and GP3.
The conditions have distinct histories: the distribution criterion
for GP1 goes back to Gelfand and Kazhdan \cite{GK72,GK75};
Gross introduced the GP2 formulation in the non-Archimedean setting
\cite{Gross}; and GP3 is the generalized Gelfand property studied by
van Dijk \cite{vanDijk}. The labels GP1--GP3 and these historical
attributions are recorded in \cite[Section 8.1]{AG}.

The method of deducing multiplicity one from symmetry of invariant
distributions also originates with Gelfand and Kazhdan
\cite{GK72,GK75}. An Archimedean form of the criterion appears in
\cite[Theorem 2.3.2]{AGS}; see also \cite[Theorem 8.1.5]{AG}.
Thus this connection predates generalized Harish--Chandra descent.
Related restriction problems are treated in
\cite{GrossPrasad,AGRS,SunZhu}.

\subsection{The proof}
We first separate the centre of $\fg$ and decompose its semisimple part
into indecomposable symmetric pairs of Lie algebras. For one such pair,
write the ambient algebra again as $\fg$ and put $X=\fg^{-\theta}$.
Let $L$ be the connected effective isotropy group and $A$ the effective
automorphism group of $(\fg,\theta)$. We prove
\begin{equation}\label{eq:intro-block}
 \{T\in\cS^*(X)^L:
       \supp T,\supp\cF T\subseteq\cN\}
 \subseteq
 \cS^*(X)^{\ker(A\to\Aut(X\sslash L))}.
\end{equation}
Here $\cN$ is the nilpotent cone and $\cF$ is Fourier transformation
for the Killing form. Section~\ref{sec:assembly} deduces regularity
from this condition, including when $G$ and $H$ are disconnected.
A finite quotient argument replaces each admissible action, modulo
$H$, by one that preserves every summand. Sums of partial Fourier
transforms then give the required support condition on each summand.

For the classical matrix representations in Section~\ref{sec:classical},
we prove that every invariant tempered distribution has the required
symmetry. The criterion in Section~\ref{sec:gaussian} reduces this
assertion to invariant theory. On polynomial multiples of a Gaussian,
infinitesimal invariance turns a quadratic moment relation into an
identity of lower polynomial degree. The classical first fundamental
theorems give invariant generators, and matrix identities verify the
criterion in the moment quotient.

Section~\ref{sec:homogeneity} uses the support conditions in
\eqref{eq:intro-block} to reduce the vanishing problem to distributions
with one of four dilation characters. Comparing these characters with
normal weights proves vanishing for the nice blocks and excludes all
but a few orbits in EIII and EVII. In EIII, stabilizer modular characters
and the action on normal symbols exclude the remaining orbits.
For EVII, the slice at orbit $20$ rules out support on orbits $20$, $21$,
and $22$. The remaining orbits $17$ and $18$ have a common boundary
orbit $15$.
A slice at this boundary reduces the problem to a $3\times3$ matrix of
normal coefficients. The residue constraint and the stabilizer equations
imply that this matrix is zero.

The automorphism calculations in Section~\ref{sec:exceptional} and
the classification in Section~\ref{sec:completion} complete the proof.
We use the classifications of nice and pleasant pairs in
\cite{Aiz,Rubio}. Appendix~\ref{app:E7} records the exceptional root
data and the exact computations used in the two slices.

\subsection{Conventions}
All algebraic groups, Lie algebras, and vector spaces are over $\C$.
Fixed subgroups of algebraic involutions are full fixed subgroups unless
a connected component is indicated. For a linear automorphism $a$, the
action on ordinary distributions is
$(a_*T)(\varphi)=T(\varphi\circ a)$. Ordinary distributions are paired
with test functions, whereas generalized functions are paired with test
densities. A positive smooth density $d\mu$ identifies an ordinary
distribution $T$ with the generalized function $F_T$ defined by
$F_T(\varphi\,d\mu)=T(\varphi)$ for compactly supported smooth
$\varphi$. We use $\ii=\sqrt{-1}$ and $|t|_\C=t\bar t$.
Fourier transformation for a nondegenerate symmetric complex bilinear
form $B$ has kernel $\exp(2\pi\ii\operatorname{Re}B(x,y))$ and the
self dual positive measure. A normal triple is written $(e,d,f)$, with
$e,f\in\fp$, $d\in\fh$, and
$[d,e]=2e$, $[d,f]=-2f$, $[e,f]=d$.

\subsection*{Acknowledgement and disclosure of AI use}
ChatGPT sol 5.6 and ChatGPT 6-Astra were used for proof exploration.
We withdraw the first version of this manuscript on arXiv because it
relied heavily on an incorrect Theorem~6.5 and misquoted a result of
Bulois in Proposition~5.12. These theorem and proposition numbers refer
to that version of our manuscript. Chen and Sun have since corrected
their paper in the updated arXiv version \cite{ChenSun}; the result
of Bulois that we misquoted is in \cite{Bulois}.

\section{Finite component actions}\label{sec:assembly}

We reduce regularity to a condition on invariant distributions whose
support and Fourier support lie in the nullcone of an indecomposable
summand. We formulate the condition using the full automorphism group
of the symmetric pair of Lie algebras. It will then apply to every
algebraic group realizing that pair.

Let $(\fg_i,\fh_i)$ be a nontrivial indecomposable semisimple symmetric
pair of Lie algebras, with involution $\theta_i$ and
$\fh_i=\fg_i^{\theta_i}$. Thus $\fg_i$ is simple, or is the sum of two
simple ideals exchanged by $\theta_i$. Put
\[
 X_i=\fg_i^{-\theta_i},\qquad
 L_i=\im\bigl(\Aut(\fg_i,\theta_i)^\circ\longrightarrow
                  \GL(X_i)\bigr),\qquad
 A_i=\Aut(\fg_i,\theta_i)|_{X_i}.
\]
The Lie algebra of $L_i$ is the image of $\fh_i$.  In particular, $L_i$
is the connected effective isotropy group in every realization of the
block.  Let $\pi_i:X_i\to Y_i=X_i\sslash L_i$ be the invariant quotient,
and set
\[
 \cN_i=\pi_i^{-1}(\pi_i(0)),\qquad
 K_i=\ker\bigl(A_i\longrightarrow\Aut(Y_i)\bigr).
\]
The nullcone $\cN_i$ is the nilpotent cone of the isotropy
representation \cite{KR}. Write $\cF_i$ for Fourier transformation with respect to the restriction
of the Killing form, the additive character
$z\mapsto\exp(2\pi\ii\operatorname{Re}z)$, and the self dual positive
measure.  Define
\begin{equation}\label{eq:C}
 \cC_i=\left\{D\in\cS^*(X_i)^{L_i}:
       \supp D\subseteq\cN_i,\quad
       \supp\cF_iD\subseteq\cN_i\right\}.
\end{equation}
The local condition used in this section is
\begin{equation}\tag{F}\label{eq:F}
 \cC_i\subseteq\cS^*(X_i)^{K_i}.
\end{equation}
For a single block we omit the subscripts and also write $K_Y=K_i$.
Several families will satisfy the stronger inclusion
\begin{equation}\label{c:eq:kernel}
 \cS^*(X_i)^{L_i}\subseteq\cS^*(X_i)^{K_i}.
\end{equation}

The following lemma relates the invariant quotients for a reductive
group and its identity component.

\begin{lemma}\label{lem:finite}
Let a reductive algebraic group $M$ act on a complex vector space $X$,
and let $L$ be the effective image of $M^\circ$.  Suppose that a linear
automorphism $u$ normalizes the effective $M$-action and preserves every
closed $M$-orbit.  There is $m\in M$ such that $m^{-1}u$ acts identically
on $X\sslash L$.
\end{lemma}
\begin{proof}
Put $Y=X\sslash L$, and let $F$ be the finite image of $M/M^\circ$ on
$Y$.  Since $u$ normalizes $L$, it induces an automorphism $\bar u$ of
$Y$.  The quotient $X\sslash M=Y\sslash F$ parametrizes closed
$M$-orbits, so the hypothesis implies that $\bar u$ induces the identity
on $Y\sslash F$.  The fibers of a finite group quotient are its orbits;
consequently
\[
 Y=\bigcup_{f\in F}\{y\in Y:\bar u(y)=f(y)\}.
\]
Each set in this union is closed.  The ring $\C[X]^L$ is a domain, so
$Y$ is irreducible, and one of these finitely many closed sets is all of
$Y$.  Lifting the corresponding element of $F$ to $m\in M$ proves the
assertion.
\end{proof}

\begin{theorem}\label{thm:assembly}
Let $G$ be a complex algebraic group with reductive identity component,
let $\theta$ be an algebraic involution, and put $H=G^\theta$.  If every
nontrivial indecomposable semisimple block of $(\fg,\theta)$ satisfies
\eqref{eq:F}, then $(G,H,\theta)$ is regular. More explicitly, write
\[
 \fp=Z\oplus\bigoplus_{i\in I}X_i,\qquad
 Z=\fz(\fg)\cap\fp,\qquad X_i=\fg_i^{-\theta}.
\]
Here the $\fg_i$ are the minimal nonzero $\theta$-stable ideals of
$[\fg,\fg]$ on which $\theta$ is nontrivial. Each is either a simple
ideal preserved by $\theta$, or the sum of two simple ideals exchanged
by $\theta$. In the latter case $X_i$ is the minus diagonal after
identifying the two ideals by $\theta$. The spaces $X_i$ are the
isotropy spaces of the indecomposable symmetric pairs; they need not
be irreducible as $L_i$-modules.
\end{theorem}
\begin{proof}
The proof has three steps. We first arrange that the admissible
automorphism preserves each $X_i$. We then establish the Fourier
support condition separately in each variable. Finally we apply
\eqref{eq:F} to one variable at a time.

\emph{Step 1: an automorphism preserving each summand.}
Use the decomposition in the statement. The simple ideals on which $\theta$ is the identity have zero minus
space.  The effective image of $H^\circ$ is
$L=\prod_{i\in I}L_i$: these connected subgroups have the same Lie
algebra and act independently on the indicated summands.  The group
$L$ acts trivially on $Z$, and hence
\[
 \fp\sslash L=Z\times\prod_{i\in I}Y_i.
\]

Fix an admissible $g\in G$, and write $u=\Ad(g)|_\fp$.  We first choose
a representative of its action modulo $H$ that preserves every block.
Lemma~\ref{lem:finite} gives $h\in H$ such that
$\beta=\Ad(h^{-1}g)|_\fp$ induces the identity on $\fp\sslash L$.
Any automorphism of $(\fg,\theta)$ preserves $Z$ and permutes its
indecomposable semisimple blocks.  The Killing quadratic form on each
$X_i$, regarded as an invariant polynomial on $\fp$, depends only on
that summand and is nonzero.  Since $\beta$ fixes each such polynomial,
it preserves every $X_i$.  It also fixes all linear functions on $Z$.
Thus
\begin{equation}\label{eq:beta}
 \beta=\id_Z\oplus\bigoplus_{i\in I}\beta_i,
 \qquad \beta_i\in K_i.
\end{equation}

Invariance and nondegeneracy of the Killing form give
$X_i^{L_i}\subseteq\fz(\fg_i)=0$. Hence
$\fp^{H^\circ}=Z$ and $\fp^H=Z^H$.  Choose an $H$-stable complement
$Z_0$ to $Z^H$ in $Z$, using the finite image of $H$ on $Z$.
The group $H^\circ$ acts trivially on $Z_0$, so the action of $H$
on $Z_0$ factors through $H/H^\circ$. This finite group action may
be nontrivial: $Z_0$ is stable as a subspace, but its vectors need
not be fixed by every component of $H$.
Equation~\eqref{eq:beta} shows that it is also $u$-stable.
We identify
\[
 V=Q_H(\fp)=Z_0\oplus\bigoplus_{i\in I}X_i.
\]
The quotient $V\sslash H$ is the finite group quotient of
$Z_0\times\prod_iY_i$ by the image of $H/H^\circ$.  This finite group
fixes the origin, whose fiber under that quotient is consequently a
single point.  Hence
\begin{equation}\label{eq:N}
 N=\Gamma_H(\fp)=\{0\}\times\prod_{i\in I}\cN_i,
 \qquad R_H(\fp)=V\setminus N.
\end{equation}

\emph{Step 2: support after partial Fourier transformation.}
Assume that every $H$-invariant Schwartz distribution on
$R_H(\fp)=V\setminus N$ is fixed by $u$, namely
\[
 \cS^*(V\setminus N)^H\subseteq\cS^*(V\setminus N)^u.
\]
We must prove the corresponding inclusion on $V$. On $H$-invariant
distributions, $u$ and $\beta$ have the same action, and admissibility
makes this action an involution.  It is therefore enough to show that
an $H$-invariant distribution $T\in\cS^*(V)$ satisfying
$\beta_*T=-T$ vanishes.  Its restriction to $V\setminus N$ is both
$\beta$-invariant by the premise and $\beta$-odd.  Thus
\begin{equation}\label{eq:suppT}
 \supp T\subseteq N.
\end{equation}

To apply \eqref{eq:F} in each variable, we prove the corresponding
support condition for every partial Fourier transform:
\begin{equation}\label{eq:partial}
 \supp\cF_iT\subseteq N\qquad(i\in I).
\end{equation}
Here $\cF_i$ acts in $X_i$ and leaves the other variables unchanged.
An individual $\cF_iT$ need not be $H$-invariant, since components
of $H$ may permute the $X_i$. We therefore first sum the transforms
over an orbit of this permutation action. The invariance of the
Killing form implies
\begin{equation}\label{eq:covariance}
 m_*\cF_i m_*^{-1}=\cF_{mi}\qquad(m\in H).
\end{equation}
Also, \eqref{eq:beta} implies that $\beta_*$ commutes with every
$\cF_i$.  For an $H$-orbit $\mathcal O\subseteq I$, the distribution
\[
 A_{\mathcal O}T=\sum_{j\in\mathcal O}\cF_jT
\]
is therefore $H$-invariant and $\beta$-odd.  The regularity premise gives
\begin{equation}\label{eq:sum}
 \supp A_{\mathcal O}T\subseteq N.
\end{equation}

We recover the individual transform from this sum by restricting to
an open set on which all other summands vanish. Fix $i\in\mathcal O$
and consider the open set
$U_i=\{(z,(x_k))\in V:x_i\notin\cN_i\}$.  If $j\ne i$, partial
Fourier transformation in $X_j$ preserves the support condition
$x_i\in\cN_i$ in \eqref{eq:suppT}.  On $U_i$ only the $i$th summand
can remain, and \eqref{eq:sum} gives
\[
 (\cF_iT)|_{U_i}=(A_{\mathcal O}T)|_{U_i}=0.
\]
The support conditions on all the untransformed variables are also
preserved, so
\[
 \supp\cF_iT\subseteq
 \{0\}\times X_i\times\prod_{k\ne i}\cN_k,
\]
where the factors retain their original order.  This proves
\eqref{eq:partial}.

\emph{Step 3: apply the block condition.}
Fix $i$ and pair $T$ with a Schwartz function $\varphi$ in all
variables other than $X_i$.  The resulting tempered distribution
$T_\varphi$ on $X_i$ is $L_i$-invariant.  Pairing commutes with
$\cF_i$, so \eqref{eq:suppT} and \eqref{eq:partial} imply
$T_\varphi\in\cC_i$.  Condition~\eqref{eq:F} gives
$(\beta_i)_*T_\varphi=T_\varphi$.  Finite sums of decomposable
Schwartz functions are dense in the Schwartz space of the product.
It follows that the automorphism acting by $\beta_i$ on $X_i$ and
trivially on every other factor fixes $T$.  Applying this to every
$i$, and using $\beta|_{Z_0}=\id$, yields $\beta_*T=T$.
Since $\beta_*T=-T$, one has $T=0$, as required.
\end{proof}

\section{A Gaussian criterion for invariant distributions}\label{sec:gaussian}
We give an algebraic criterion for an involution to fix all invariant
tempered distributions. The criterion concerns its action on the
invariants in the cotangent moment quotient. The proof uses polynomial
multiples of a Gaussian as test functions and induction on polynomial
degree.
\subsection{Gaussian test functions}
\begin{lemma}[Gaussian totality]\label{lem:totality}
For any Euclidean norm on $\R^N$, the space
\[
 \C[x_1,\ldots,x_N]e^{-\lVert x\rVert^2/2}
\]
is dense in $\cS(\R^N)$ for its Schwartz topology.
\end{lemma}
\begin{proof}
By the Hahn--Banach theorem it suffices to show that a continuous linear functional $T$ annihilating this space is zero.  Put $\gamma(x)=e^{-\lVert x\rVert^2/2}$ and
\[
 F(w)=T\bigl(\gamma(x)e^{w\cdot x}\bigr),\qquad w\in\C^N.
\]
The map $w\mapsto\gamma e^{w\cdot x}$ is entire with values in Schwartz space.  Indeed, on each compact set of $w$'s, every $x$-derivative and every $w$-derivative is bounded in every Schwartz seminorm by a polynomial in $|x|$ times
\[
 e^{-|x|^2/2+C|x|},
\]
for a fixed $C$.  Difference quotients and their derivatives therefore converge in all Schwartz seminorms.  Consequently $F$ is entire, and
\[
 \partial_w^\alpha F(0)=T(x^\alpha\gamma)=0
\]
for every multiindex $\alpha$.  Thus $F=0$.

For $w=-\ii\xi$, Schwartz seminorms of $\gamma e^{-\ii\xi\cdot x}$ grow at most polynomially in $\xi$.  Integration against a Schwartz function of $\xi$ is therefore legitimate in Schwartz space, and shows that the Fourier transform of $\gamma T$ is the function $F(-\ii\xi)$, with the matching Fourier convention.  It is zero.  Injectivity of the Fourier transform on tempered distributions gives $\gamma T=0$.  Since $\gamma$ is nowhere zero, for every $f\in C_c^\infty(\R^N)$,
\[
 T(f)=(\gamma T)(f/\gamma)=0.
\]
Thus $T=0$, and Hahn--Banach proves the assertion.
\end{proof}

\subsection{The moment map criterion}
Let $L$ be a connected complex reductive group acting algebraically
on a finite dimensional complex vector space $X$, and put $\fl=\Lie_\C L$.
Fix a compact real form $K$ with Lie algebra $\fk$; thus
$\fl=\fk\oplus\ii\fk$ as real vector spaces and $K$ is Zariski
dense in $L$. Suppose that a complex linear involution $s$ of $X$
normalizes both actions. Averaging over the compact group generated
by the effective image of $K$ and $s$ gives a positive Hermitian
form $h$ invariant under both.

The cotangent moment map is the map
$\mu:T^*X=X\oplus X^*\to\fl^*$ whose pairing with $A\in\fl$ is
$\mu_A(x,\xi)=\langle\mu(x,\xi),A\rangle=\xi(Ax)$.
We call $\mu_A$ the \emph{infinitesimal moment pairing}. Put
\begin{equation}\label{eq:moment-general}
 \mu_A(x,\xi)=\xi(Ax)\quad(A\in\fl),\qquad
 \Palg=\C[T^*X],\qquad
 \Ialg=(\mu_A:A\in\fl),\qquad
 \Ralg=\Palg/\Ialg.
\end{equation}
The ideal $\Ialg$ is homogeneous and generated by quadratics.
The cotangent action preserves it, since
$g\cdot\mu_A=\mu_{\Ad(g)A}$. The lift
$(x,\xi)\mapsto(sx,\xi\circ s^{-1})$ preserves it as well.
Thus $s$ acts on the invariant algebra $\Ralg^L$.

\begin{theorem}[Gaussian criterion]\label{thm:gaussian}
With the preceding data, assume that $s$ acts identically on $\Ralg^L$.  Then
\begin{equation}\label{eq:gaussian-conclusion}
 \cS^*(X)^{L,s=-1}=0.
\end{equation}
Equivalently, every $L$-invariant tempered distribution is $s$-invariant.
\end{theorem}

\begin{proof}
Choose $h$-orthonormal coordinates $z=(z_1,\ldots,z_N)^t$ on $X$
and dual coordinates $w$ on $X^*$, where $N=\dim_\C X$.
Write $\bar z=(\overline{z_1},\ldots,\overline{z_N})^t$ for
coordinatewise complex conjugation. Thus $w=\bar z$ denotes the
covector $v\mapsto\bar z^t v$. Put
\[
 \gamma(z)=e^{-\bar z^tz/2},\qquad
 \widetilde P(z)=P(z,\bar z)\quad(P\in\Palg).
\]
The substitution $P\mapsto\widetilde P$ is an algebra isomorphism
onto the complex valued polynomials on the underlying real space
of $X$, and it preserves degree. It is equivariant for $K$ and $s$:
for a unitary matrix $k$,
\[
 \widetilde{k\cdot P}(z)
 =P(k^{-1}z,k^t\bar z)
 =P(k^{-1}z,\overline{k^{-1}z}).
\]

For $A\in\ii\fk$, the matrix $A$ is Hermitian. Define
\[
 D_Af(z)=\left.\frac{d}{dt}\right|_{t=0}f(e^{-tA}z).
\]
The parameter $t$ is real. Since $A^*=A$,
\[
 \left.\frac{d}{dt}\right|_{t=0}\|e^{-tA}z\|_h^2
 =-\bar z^t(A^*+A)z=-2\bar z^tAz.
\]
The chain rule for $\exp(-\|z\|_h^2/2)$ therefore gives the first
identity below, because $\widetilde{\mu_A}(z)=\bar z^tAz$.
The second identity is the product rule:
\begin{equation}\label{eq:gaussian-derivative}
 D_A\gamma=\widetilde{\mu_A}\gamma,\qquad
 D_A(\widetilde Q\gamma)
   =(D_A\widetilde Q+\widetilde{\mu_A}\widetilde Q)\gamma.
\end{equation}
Let $T\in\cS^*(X)^{L,s=-1}$. Differentiating its invariance in
Schwartz space yields
\begin{equation}\label{eq:moment-lowering}
 T(\widetilde{\mu_A}\widetilde Q\gamma)
   =-T((D_A\widetilde Q)\gamma).
\end{equation}

We show by induction on $d$ that $T(\widetilde P\gamma)=0$ whenever
$\deg P\leq d$. Average $P$ over $K$ and take its odd part:
\[
 P_K=\int_K k\cdot P\,dk,\qquad P^-=\frac{P_K-s\cdot P_K}{2}.
\]
The equivariance of substitution and the invariance of $\gamma$ give
\[
 T(\widetilde P\gamma)=T(\widetilde{P^-}\gamma).
\]
Since $K$ is Zariski dense in $L$, the polynomial $P^-$ is
$L$-invariant. Its class in $\Ralg^L$ is both fixed and negated
by $s$, so $P^-\in\Ialg$.

Choose a real basis $A_1,\ldots,A_\ell$ of $\ii\fk$, which is also
a complex basis of $\fl$. Since $\Ialg$ is generated by the
homogeneous quadratics $\mu_{A_j}$, comparison of homogeneous
components gives
\begin{equation}\label{eq:ideal-expression}
 P^-=\sum_{j=1}^{\ell}\mu_{A_j}Q_j,\qquad \deg Q_j\leq d-2.
\end{equation}
For $d<2$ this means $P^-=0$ and proves the initial cases.
For $d\geq2$, \eqref{eq:moment-lowering} gives
\[
 T(\widetilde P\gamma)
   =-\sum_jT((D_{A_j}\widetilde Q_j)\gamma)=0
\]
by induction: the linear vector field $D_{A_j}$ does not increase
total polynomial degree, so
\[
 \deg(D_{A_j}\widetilde Q_j)\leq\deg\widetilde Q_j
 =\deg Q_j\leq d-2
\]
whenever this derivative is nonzero. Thus every term is covered by
an earlier induction step. Lemma~\ref{lem:totality} gives $T=0$.
Applying this to $(T-s_*T)/2$ proves the equivalent formulation
for an arbitrary $L$-invariant tempered distribution.
\end{proof}

\subsection{Differential operators and the moment map}
The criterion has a natural interpretation in terms of modules over
the algebra $\DD(X)$ of polynomial differential operators. For
$A\in\fl$, let
\[
 \xi_A=\sum_j(Az)_j\partial_{z_j}.
\]
Here $A$ denotes the matrix of the infinitesimal action in the chosen
coordinates, and $(Az)_j=\sum_k A_{jk}z_k$ is the $j$th coordinate
of the vector $Az$. The principal symbol of $\xi_A$ for the order
filtration is $\mu_A$.
Thus the right $\DD(X)$-module
\[
 \mathcal M_L=\DD(X)\big/\sum_{A\in\fl}\xi_A\DD(X)
\]
has an associated graded module which is a quotient of $\Palg/\Ialg$.
In particular, its characteristic variety lies in $\mu^{-1}(0)$.
Polynomial differential operators act on ordinary distributions on the right by
$(T\cdot P)(\varphi)=T(P\varphi)$.
Infinitesimal $L$-invariance is the system
$T\cdot\xi_A=T\cdot\overline{\xi_A}=0$;
the equations for $A$ and $\ii A$ separate the holomorphic and
antiholomorphic parts. The module $\mathcal M_L$ records the
holomorphic part of this system.

The proof above realizes these equations on polynomial multiples
of a Gaussian. On the complexification of the underlying real
space, these functions form an exponential module for polynomial
differential operators. Conjugation by $\gamma$ gives
\[
 \gamma^{-1}D_A\gamma=D_A+\widetilde{\mu_A}
 \qquad(A\in\ii\fk).
\]
Multiplication by $\widetilde{\mu_A}$ raises polynomial degree by
two, whereas $D_A$ preserves it. The induction is therefore a
filtered argument on the coinvariants of this exponential module.
The hypothesis on $\Ralg^L$ removes the odd invariant classes at
each degree, and Gaussian totality passes from polynomial test
functions to Schwartz functions. This is the connection with
$\DD$-modules; neither holonomicity nor an equality of characteristic
schemes is required.
This use of polynomial differential operators is also consistent
with the operator form of classical invariant theory in
\cite[\S\S2.3, 3.5, 3.8.4 and Appendix 2]{Howe}.

The distinction between $\Ialg$ and its radical is essential.
Triviality of $s$ only on the reduced zero fibre does not give
\eqref{eq:ideal-expression}. All applications below establish
identities in $\Palg/\Ialg$ itself.

\subsection{Invariant representatives}
\begin{lemma}[Invariant representatives]\label{lem:representatives}
The natural map
\begin{equation}\label{eq:invariant-surjection}
 \Palg^L\longrightarrow(\Palg/\Ialg)^L=\Ralg^L,
 \qquad P\longmapsto[P],
\end{equation}
is surjective. Equivalently, every invariant quotient class has an
$L$-invariant polynomial representative.
\end{lemma}
\begin{proof}
Let $[P]\in(\Palg/\Ialg)^L$. Since $K\subset L$, one has
$[k\cdot P]=[P]$ for every $k\in K$. Take Haar measure $dk$ of total
mass one on $K$. Averaging in the finite dimensional space of
polynomials of degree at most $\deg P$ gives
\[
 P_K=\int_K k\cdot P\,dk,\qquad
 [P_K]=\int_K[k\cdot P]\,dk=[P].
\]
Here the quotient map commutes with the finite dimensional integral.
The polynomial $P_K$ is $K$-invariant and hence $L$-invariant by Zariski
density. It is the required representative.
\end{proof}

\section{Classical symmetric blocks}\label{sec:classical}
We establish the block condition \eqref{eq:F} for AII, unbalanced AIII,
CII, DIII, and unequal orthogonal blocks. In these cases we prove
that every invariant tempered distribution is fixed by the kernel
of the action on the invariant quotient. In each matrix model we first determine
invariant generators, then prove that the indicated symmetries act
identically on $\Ralg^L$, and finally apply the Gaussian criterion.
We subsequently identify the automorphisms of each block.
The nice classical blocks are treated in
Section~\ref{sec:nice}.

\subsection{Invariant tensors}
We use the classical first fundamental theorems in the form discussed
by Howe \cite[Theorems 2.2.2, 3.4.4 and 3.8.3.2]{Howe}.
For a general linear group, invariant tensors are generated by
evaluations between vectors and covectors. For an orthogonal or
symplectic group they are generated by pairings with the defining
bilinear form. The statements for tensors follow from their polynomial
forms by polarization. For a special orthogonal group one also uses
a volume tensor; this is the determinant component described after
\cite[Corollary 3.3.2]{Howe}.

The second fundamental theorem concerns relations between these
generators. For example, if $\dim V=d$, the relations between
$c_{ij}=\lambda_j(v_i)$ are generated by the minors of size $d+1$
of $(c_{ij})$; see \cite[\S2.6]{Howe}. Thus contractions in small
dimension need not be independent. Our arguments use their images
in the polynomial ring, so these dimension relations are already
imposed before taking the moment quotient.
For special orthogonal vertices we shall also use the determinant
identity
\begin{equation}\label{eq:volume-determinant}
 \varepsilon(v_1,\ldots,v_d)\varepsilon(w_1,\ldots,w_d)
   =\det\bigl(B(v_i,w_j)\bigr)_{1\leq i,j\leq d},
\end{equation}
where the volume tensor is normalized in an orthonormal basis.
Indeed, both sides equal $\det(V)\det(W)$ in that basis.
This replaces two volume tensors at one vertex by pairings, and
leaves at most one volume tensor there. We need this explicit
relation, rather than a presentation of every classical invariant ring.

Applied at each vertex, these contraction rules give the generators
in the mixed quiver theorem of
\cite[Theorem 1 and Corollary 1]{Lopatin}. We retain that theorem
for the precise arrow conventions. In each model we describe the
contractions and compute their images in the moment quotient.

For matrix variables, a contraction using only pairings is a product
of traces along closed paths. To see this, associate two indices with
each matrix entry and join each contracted pair of indices. Every entry
then belongs to a cycle, and the summation around that cycle is the
trace of the corresponding matrix product. At a symplectic or orthogonal
vertex, reversing an arrow inserts the adjoint for the defining form.
Skew symmetric arrows joining dual general linear spaces contribute
their transposes, which differ only by a sign. Thus the same description
applies to them. Polarization and substitution of repeated variables
recover all polynomial invariants from these multilinear contractions.

For an arrow taking values in $\bigwedge^2 V$, use the equivariant
projection
\[
 \operatorname{Alt}:V\otimes V\longrightarrow\bigwedge^2 V,
 \qquad
 v\otimes w\longmapsto\tfrac12(v\otimes w-w\otimes v),
\]
where $\bigwedge^2 V$ is identified with the alternating tensors.
An invariant polynomial on the alternating tensors extends to
$V\otimes V$ by composition with $\operatorname{Alt}$.
Applying the tensor contractions to this extension and then
restricting back gives the required contractions for the
skew symmetric arrow. The same argument applies to
$\bigwedge^2 V^*$ and to several such arrows simultaneously.

Write $\det(I+tM)=\sum_{k\geq0}\sigma_k(M)t^k$.
Newton identities give
\[
 k\sigma_k(M)=
 \sum_{j=1}^k(-1)^{j-1}\sigma_{k-j}(M)\tr(M^j).
\]
Over $\C$, induction expresses every characteristic coefficient
in traces of powers. A power of a closed path product is again
a closed path product, so we use traces throughout.
Lemma~\ref{lem:representatives} shows that the images of generators of
$\Palg^L$ generate $\Ralg^L$. We therefore state the generator calculation
separately from the calculation of each symmetry on the quotient.

\subsection{The paired general linear models}
There are two models with the same moment relations. In the first, let
\[
 X=\Hom(W,U)\oplus\Hom(U,W),\qquad L=\GL(U)\times\GL(W),
\]
with action $(a,b)(B,C)=(aBb^{-1},bCa^{-1})$ and involution
$\tau(B,C)=(C^t,B^t)$. In the second, let
\[
 X=\Lambda^2V\oplus\Lambda^2V^*,\qquad L=\GL(V),
\]
represented by skew symmetric matrices, with action
$a(B,C)=(aBa^t,a^{-t}Ca^{-1})$ and involution $s(B,C)=(C,B)$.
The effective images may be used in place of the displayed groups.

Identify the cotangent space with two copies of $X$ using the
nondegenerate invariant symmetric pairing
\[
 \langle(B,C),(D,E)\rangle=\tr(BE)+\tr(CD).
\]
For rectangular matrices the two traces are taken on $U$ and $W$.
For $A=(a,b)\in\mathfrak{gl}(U)\oplus\mathfrak{gl}(W)$,
the infinitesimal moment pairing $\mu_A$ of
\eqref{eq:moment-general} is
\[
 \tr_U\bigl(a(BE-DC)\bigr)+\tr_W\bigl(b(CD-EB)\bigr)
\]
in the first model. In the skew symmetric model all four matrices are
skew symmetric, and the moment pairing is
\[
 \tr\bigl((aB+Ba^t)E+(-a^tC-Ca)D\bigr)
 =2\tr\bigl(a(BE-DC)\bigr).
\]
Consequently the moment ideals are, respectively,
\begin{equation}\label{eq:paired-moment}
 \Ialg=(BE-DC,\ CD-EB),\qquad \Ialg=(BE-DC),
\end{equation}
where a matrix inside parentheses denotes all of its entries. In the
second model the transpose of the defining relation is $EB=CD$.
In each model put $\Palg=\C[T^*X]$ and $\Ralg=\Palg/\Ialg$.

\begin{proposition}[Invariant generators for the paired models]
\label{prop:paired-generators}
In both models the invariant algebra $\Palg^L$ is generated by
\begin{equation}\label{eq:paired-traces}
 \tr(A_1F_1\cdots A_kF_k),\qquad
 A_j\in\{B,D\},\quad F_j\in\{C,E\},\quad k\geq1.
\end{equation}
For rectangular matrices the trace is taken on $U$.
\end{proposition}
\begin{proof}
In the first model, apply the quiver trace theorem of Le Bruyn and
Procesi \cite{LBP}, equivalently the contractions from
\cite[Theorem 2.2.2]{Howe} at the two vertices $U,W$, to the arrows
\[
 B,D:W\longrightarrow U,\qquad C,E:U\longrightarrow W.
\]
Every closed path alternates between the two directions. A cycle based
at $U$ has the form \eqref{eq:paired-traces}; a cycle based at $W$ is
moved to $U$ by cyclicity of rectangular trace. This gives the claimed
generators in the first model.

In the second model, apply \cite[Corollary~1]{Lopatin} to two dual
vertices $V,V^*$ with arrows
\[
 B,D:V^*\longrightarrow V,\qquad C,E:V\longrightarrow V^*,
\]
all of skew symmetric type $S^-$ in the terminology of
\cite[\S2.1]{Lopatin}. The vertex involution exchanges $V,V^*$,
and the vertex group elements are $a,a^{-t}$, so the acting group
is $\GL(V)$. The induced actions are $A\mapsto aAa^t$ for
$A\in\{B,D\}$ and $F\mapsto a^{-t}Fa^{-1}$ for $F\in\{C,E\}$,
as required. These arrows satisfy the condition that a
skew symmetric arrow joins dual vertices; no additional arrows
are introduced for them in the double quiver.

Every closed path alternates between the two directions. After
a cyclic permutation, its trace is therefore
$\tr(A_1F_1\cdots A_kF_k)$ with $A_j\in\{B,D\}$ and
$F_j\in\{C,E\}$. Lopatin's corollary and Newton identities give
the claimed generators.
\end{proof}

\begin{proposition}[The paired involutions on the invariant quotient]
\label{prop:paired-ring}
The involution $\tau$ in the first model and the swap $s$ in the second
model act identically on $\Ralg^L$.
\end{proposition}
\begin{proof}
By Proposition~\ref{prop:paired-generators} and
Lemma~\ref{lem:representatives}, it suffices to fix the classes of the
alternating traces \eqref{eq:paired-traces}. In $\Ralg$, the moment
relations give
\[
 BE=DC,\qquad EB=CD.
\]
Set
\[
 P=BC,\qquad Q=BE=DC,\qquad S=DE.
\]
The following matrix identities hold over this ring:
\begin{align*}
 PQ&=BCDC=BEBC=QP,\\
 QS&=DCDE=DEBE=SQ,\\
 PS&=BCDE=BEBE=Q^2,\\
 SP&=DEBC=DCDC=Q^2.
\end{align*}
Thus $P,Q,S$ commute. Every trace in
Proposition~\ref{prop:paired-generators} is a trace of a word in these
three matrices. The cotangent lifts of the two involutions are,
respectively,
\[
 (B,C,D,E)\longmapsto(C^t,B^t,E^t,D^t),\qquad
 (B,C,D,E)\longmapsto(C,B,E,D).
\]
Each sends $P,Q,S$ to $P^t,Q^t,S^t$, using the moment relation for $Q$.
Transposition reverses the order of each word inside its trace.
Since $P,Q,S$ commute in $\Ralg$, the trace is unchanged. Hence all
invariant quotient classes are fixed.
\end{proof}

\begin{proposition}\label{prop:paired-linear}
Every $L$-invariant tempered distribution in the first model is
$\tau$-invariant, and every such distribution in the second model is
$s$-invariant.
\end{proposition}
\begin{proof}
The compact real forms are $\U(U)\times\U(W)$ in the first model and
$\U(V)$ in the second. The indicated involutions normalize the complex
groups and their compact forms, and preserve the sum of the Frobenius
Hermitian norms. Proposition~\ref{prop:paired-ring} verifies the invariant
quotient hypothesis of Theorem~\ref{thm:gaussian}. Applying that theorem
in each model proves the assertion.
\end{proof}

\subsection{Balanced symplectic blocks}
For symplectic spaces $U,W$ of the same dimension, fix a symplectic
identification between them. Set
\[
 X=\Hom(W,U),\quad L=\Sp(U)\times\Sp(W),\quad
 (a,b)C=aCb^{-1},\quad s(C)=-C^*.
\]
Here $C^*:U\to W$ denotes the symplectic adjoint, determined by
$\omega_U(Cw,u)=\omega_W(w,C^*u)$.

The nondegenerate symmetric invariant pairing
$\langle C,D\rangle=\tr_U(CD^*)$ identifies $T^*X$ with $X\times X$.
For $a\in\mathfrak{sp}(U)$ and $b\in\mathfrak{sp}(W)$,
\[
 \langle aC-Cb,D\rangle
 =\tfrac12\tr_U\bigl(a(CD^*-DC^*)\bigr)
  -\tfrac12\tr_W\bigl(b(D^*C-C^*D)\bigr).
\]
Nondegeneracy of the trace forms on the symplectic Lie algebras gives
\begin{equation}\label{eq:CII-moment}
 \Ialg=(CD^*-DC^*,\ C^*D-D^*C),
\end{equation}
where the entries of the displayed matrices generate the ideal.
Put $X_0=C$, $X_1=D$, $\Palg=\C[T^*X]$, and $\Ralg=\Palg/\Ialg$.

\begin{proposition}[Invariant generators for the symplectic model]
\label{prop:CII-generators}
The invariant algebra $\Palg^L$ is generated by the cycle traces
\begin{equation}\label{eq:cycle-traces}
 t_{i_1\ldots i_{2k}}
 =\tr_U(X_{i_1}X_{i_2}^*\cdots X_{i_{2k-1}}X_{i_{2k}}^*),
 \qquad i_j\in\{0,1\},\quad k\geq1.
\end{equation}
\end{proposition}
\begin{proof}
Apply \cite[Corollary~1]{Lopatin} to the symplectic vertices
$U,W$ and the two unrestricted arrows $C,D:W\to U$.
The acting group is $\Sp(U)\times\Sp(W)$ with the given action.
For the standard symplectic matrices, $J_U^2=J_W^2=-I$ and
\[
 Y^*=J_W^{-1}Y^tJ_U,\qquad J_WY^tJ_U=-Y^*.
\]
Thus the reverse arrow in Lopatin's double quiver evaluates to
$-Y^*$ for $Y=C,D$. Every closed path alternates between a map
$W\to U$ and a reverse map $U\to W$. Its trace is consequently,
up to sign and a cyclic permutation,
\[
 \tr_U(Y_1Y_2^*\cdots Y_{2k-1}Y_{2k}^*),\qquad Y_j\in\{C,D\}.
\]
Cyclicity of rectangular traces also puts paths based at $W$
in this form. The corollary and Newton identities now yield
\eqref{eq:cycle-traces}.
\end{proof}

\begin{proposition}[The symplectic involution on the invariant quotient]
\label{prop:CII-ring}
The involution $s$ acts identically on $\Ralg^L$.
\end{proposition}
\begin{proof}
It suffices to fix the cycle traces of
Proposition~\ref{prop:CII-generators}, by
Lemma~\ref{lem:representatives}. In $\Ralg$, the moment relations give,
for $i,j\in\{0,1\}$,
\begin{equation}\label{eq:adjacent-exchanges}
 X_iX_j^*=X_jX_i^*,\qquad X_i^*X_j=X_j^*X_i.
\end{equation}
In an alternating matrix word, the first equality interchanges adjacent
indices in positions $(2j-1,2j)$, while the second interchanges those in
positions $(2j,2j+1)$. Therefore any permutation of the indices leaves
the word unchanged when the alternating star pattern is kept fixed.

The cotangent lift sends $(C,D)$ to $(-C^*,-D^*)$. In an even cycle the
minus signs cancel, and cyclicity of rectangular traces gives
\begin{align*}
 s(t_{i_1\ldots i_{2k}})
 &=\tr_W(X_{i_1}^*X_{i_2}\cdots X_{i_{2k-1}}^*X_{i_{2k}})\\
 &=\tr_U(X_{i_2}X_{i_3}^*\cdots X_{i_{2k}}X_{i_1}^*)\\
 &=t_{i_1\ldots i_{2k}}.
\end{align*}
The last equality uses \eqref{eq:adjacent-exchanges}. Thus every
generator class is fixed, proving the assertion.
\end{proof}

\begin{proposition}\label{prop:symplectic-linear}
Every $L$-invariant tempered distribution on $X$ is $s$-invariant.
\end{proposition}
\begin{proof}
Choose symplectic bases with the same standard real form matrix. Then
$s$ is a complex linear unitary involution for the Frobenius Hermitian
norm and normalizes $L$ and its compact real form
$\USp(U)\times\USp(W)$. Proposition~\ref{prop:CII-ring} and
Theorem~\ref{thm:gaussian} give the assertion.
\end{proof}

\subsection{Unequal orthogonal blocks}
Let $U,W$ be quadratic spaces of dimensions $p,q>0$, with $p\ne q$.
We use orthonormal bases and the action
\[
 X=\Hom(W,U),\qquad L=\SO(U)\times\SO(W),\qquad
 (a,b)C=aCb^{-1}.
\]
Identify $T^*X=X\times X$ by the pairing
$\langle C,D\rangle=\tr(CD^t)$. For $a\in\mathfrak{so}(U)$ and
$b\in\mathfrak{so}(W)$, the moment pairing is
\[
 \langle aC-Cb,D\rangle
 =\tfrac12\tr\bigl(a(CD^t-DC^t)\bigr)
  -\tfrac12\tr\bigl(b(D^tC-C^tD)\bigr).
\]
Nondegeneracy of the trace forms gives
\begin{equation}\label{eq:O-moment}
 \Ialg=(CD^t-DC^t,\ C^tD-D^tC),
\end{equation}
again taking all entries as generators. For a one dimensional factor the
corresponding skew symmetric matrix is zero. Set $X_0=C$, $X_1=D$,
$\Palg=\C[T^*X]$, and $\Ralg=\Palg/\Ialg$.
The exchange identities \eqref{eq:adjacent-exchanges} hold with $*=t$.

\begin{proposition}[Invariant generators for the unequal orthogonal quotient]
\label{prop:O-generators}
The algebra $\Ralg^L$ is generated by the classes of the cycle traces
\eqref{eq:cycle-traces}, with $*=t$.
\end{proposition}
\begin{proof}
Apply \cite[Theorem~1]{Lopatin} to two special orthogonal vertices
$U,W$ and two unrestricted arrows $C,D:W\to U$. The vertex involution
fixes both vertices, and the reverse arrows evaluate to $C^t,D^t$.
The theorem gives characteristic coefficients of products along
closed paths and volume generators. Relation
\eqref{eq:volume-determinant} reduces the latter to at most one
volume form at each vertex.
The former may be replaced by the cycle traces
\eqref{eq:cycle-traces}, using Newton identities and cyclicity of
rectangular traces.

To write the volume generators explicitly, the matrix products between
the two vertices have the following forms:
\[
 \begin{aligned}
 M&=Y_1Y_2^t\cdots Y_{2k-1}Y_{2k}^t\in\End(U),&& k\geq1,\\
 N&=Y_1^tY_2\cdots Y_{2k-1}^tY_{2k}\in\End(W),&& k\geq1,\\
 Q&=Y_1Y_2^t\cdots Y_{2k}^tY_{2k+1}\in\Hom(W,U),&& k\geq0,
 \end{aligned}
 \qquad Y_j\in\{C,D\}.
\]
Products in the other direction are transposes of those of type $Q$.
In the definition of the volume generators
\cite[\S3.2, formula~(2)]{Lopatin}, each matrix entry uses two volume
positions. Entries within the $U$-volume form are of type $M$, those
within the $W$-volume form of type $N$, and those between the two forms
of type $Q$. Ordering these positions gives, up to nonzero scalar
factors, the single volume generators
\[
 \begin{aligned}
 F_U&=\sum_{a_1,\ldots,a_p=1}^p
 \varepsilon^U_{a_1\cdots a_p}
 \prod_{h=1}^{p/2}(M_h)_{a_{2h-1}a_{2h}}
 &&(p\text{ even}),\\
 F_W&=\sum_{b_1,\ldots,b_q=1}^q
 \varepsilon^W_{b_1\cdots b_q}
 \prod_{h=1}^{q/2}(N_h)_{b_{2h-1}b_{2h}}
 &&(q\text{ even}),
 \end{aligned}
\]
and the two volume generators
\begin{equation}\label{eq:O-volume-generators}
 \begin{aligned}
 F_{U,W}={}&\sum_{\mathbf a,\mathbf b}
 \varepsilon^U_{a_1\cdots a_p}\varepsilon^W_{b_1\cdots b_q}
 \prod_{h=1}^{\alpha}(M_h)_{a_{2h-1}a_{2h}}
 \prod_{h=1}^{\beta}(N_h)_{b_{2h-1}b_{2h}}\\
 &\hspace{15mm}{}\times
 \prod_{\ell=1}^{t}(Q_\ell)_{a_{2\alpha+\ell}b_{2\beta+\ell}},
 \qquad 2\alpha+t=p,\quad 2\beta+t=q.
 \end{aligned}
\end{equation}
Here $\alpha,\beta,t\geq0$, each $a_i$ ranges from $1$ to $p$ and each
$b_j$ from $1$ to $q$, and the alternating symbols are normalized by
$\varepsilon^U_{1\cdots p}=\varepsilon^W_{1\cdots q}=1$.
The counting identities express that each volume position occurs in
exactly one matrix entry. Empty products are $1$.

By Lemma~\ref{lem:representatives}, the images of these polynomial
generators generate $\Ralg^L$. In $\Ralg$ the relations
\[
 AB^t=BA^t,\qquad A^tB=B^tA,\qquad A,B\in\{C,D\},
\]
exchange adjacent matrix choices while keeping the transpose pattern
fixed. In particular,
\[
 M^t=Y_{2k}Y_{2k-1}^t\cdots Y_2Y_1^t
     =Y_1Y_2^t\cdots Y_{2k-1}Y_{2k}^t=M,\qquad N^t=N.
\]
An $M$-entry contracted against two positions in the $U$-volume form
therefore gives
\[
 \sum_{a,b}\varepsilon^U_{ab\,a_3\cdots a_p}M_{ab}
 =\frac12\sum_{a,b}\varepsilon^U_{ab\,a_3\cdots a_p}
 (M_{ab}-M_{ba})=0.
\]
All other entries are independent of these two indices. The same
calculation applies to an $N$-entry and the $W$-volume form.
Thus $F_U,F_W$ vanish, and a nonzero image of
\eqref{eq:O-volume-generators} requires $\alpha=\beta=0$.
Its counting identities then give $p=q=t$, contrary to $p\ne q$.
Every volume generator therefore vanishes in $\Ralg$, leaving exactly
the cycle traces in the statement.
\end{proof}

\begin{proposition}[The orthogonal action on the invariant quotient]
\label{prop:O-ring}
The full group $\Or(U)\times\Or(W)$ acts identically on $\Ralg^L$.
\end{proposition}
\begin{proof}
By Proposition~\ref{prop:O-generators}, it suffices to fix the cycle
traces. Under $(C,D)\mapsto(aCb^{-1},aDb^{-1})$, with
$(a,b)\in\Or(U)\times\Or(W)$, each cycle product transforms as
\[
 Y_1Y_2^t\cdots Y_{2k-1}Y_{2k}^t
 \longmapsto a(Y_1Y_2^t\cdots Y_{2k-1}Y_{2k}^t)a^{-1}.
\]
Its trace is invariant under the full orthogonal product. Hence every
generator class, and therefore every element of $\Ralg^L$, is fixed.
\end{proof}

\begin{proposition}\label{prop:orthogonal-linear}
One has
\begin{equation}\label{eq:orthogonal-invariance}
 \cS^*(X)^{\SO(U)\times\SO(W)}
 =\cS^*(X)^{\Or(U)\times\Or(W)}.
\end{equation}
\end{proposition}
\begin{proof}
Choose a real orthogonal reflection in each factor. Each reflection is
a complex linear involution normalizing $L$ and the compact form
$\SO(p,\R)\times\SO(q,\R)$, and preserves the Frobenius Hermitian norm.
Proposition~\ref{prop:O-ring} verifies the invariant quotient hypothesis
of Theorem~\ref{thm:gaussian} for both reflections. Apply the theorem to
each reflection separately. Together with $L$, these two reflections
generate $\Or(U)\times\Or(W)$, proving
\eqref{eq:orthogonal-invariance}.
\end{proof}

\subsection{Automorphisms of the classical blocks}
For a semisimple block, retain the notation
$A=\Aut(\fg,\theta)|_X$ and $K_Y=\ker(A\to\Aut(X\sslash L))$
from Section~\ref{sec:assembly}.
The centralizers of involutions used below are given by Huang and Yu
\cite[Section 5, Table 2]{HY}. Their table is stated for a compact
simple Lie algebra $\mathfrak u$. Choose a $\theta$-stable compact
real form $\mathfrak u\subset\fg$. Polar decomposition identifies
$\Aut(\mathfrak u)^\theta$ as a maximal compact subgroup of
$\Aut(\fg)^\theta$ and identifies their component groups. Indeed,
uniqueness of the compact and positive factors makes both commute
with $\theta$. Complexifying the identity component and the displayed
representatives therefore gives the complex centralizer. We record
the matrix actions needed to apply the Gaussian criterion.

\begin{proposition}\label{prop:classical}
Every $L$-invariant tempered distribution is fixed by $K_Y$ for AII
with $n\geq2$, unbalanced AIII, CII, DIII with $n\geq3$, and BDI
with $p\ne q$. In particular, these blocks satisfy \eqref{eq:F}.
\end{proposition}
\begin{proof}
We determine $A$ in each family and apply the preceding propositions
to the elements that act trivially on the invariant quotient.

\emph{AII.} Let $V$ be symplectic of dimension $2n$. Here
\[
 \theta(x)=-x^*,\qquad
 X=\{x\in\End(V):x^*=x,\ \tr x=0\},\qquad L=\im\Sp(V).
\]
An inner automorphism represented by $a\in\GL(V)$ commutes with
$\theta$ exactly when $a^*a$ is scalar. Rescaling $a$ makes it
symplectic, so the effective inner centralizer is $L$. The involution
$\theta$ represents the outer class of $\mathfrak{sl}_{2n}$ and acts
as $-1$ on $X$. Thus $A=\langle L,-1\rangle$. For $n\geq3$, the
invariant $\tr(x^3)$ is nonzero and odd: at
$x=\diag(D,D)$, $D=\diag(1,1,-2,0,\ldots,0)$, it is $-12$.
It follows that $K_Y=L$.

When $n=2$, $X$ has dimension five, and $\tr(xy)$ is nondegenerate.
The action of $\mathfrak{sp}_4$ on $X$ is nonzero and hence injective,
and both $\mathfrak{sp}_4$ and $\mathfrak{so}(X)$ have dimension ten.
Thus its connected image is $\SO(X)$. Proposition~\ref{prop:orthogonal-linear}
for dimensions $(5,1)$ gives invariance under $\Or(X)$, including $-1$.

\emph{Unbalanced AIII.} Let $\dim U=p<q=\dim W$ and use the first
model of Proposition~\ref{prop:paired-linear}. The subgroup
$S(\GL(U)\times\GL(W))$ has the same effective action: a common
scalar multiplier adjusts the product of determinants without changing
the action on $X$. Write the involution as conjugation by
$d=\diag(I_U,-I_W)$. An inner automorphism centralizing it has a
representative $a$ with $ada^{-1}=d$ or $-d$. Unequal eigenspace
dimensions exclude the second alternative, so the inner action is $L$.
Negative transpose commutes with the involution and restricts to
$-\tau$ on $X$. Since $-1\in L$, one has $A=\langle L,\tau\rangle$.
Proposition~\ref{prop:paired-linear} gives invariance under $A$.

\emph{CII.} Let $\dim U=2p$, $\dim W=2q$, with both spaces
symplectic. The isotropy space is represented by
\[
 e_C=\begin{pmatrix}0&C\\-C^*&0\end{pmatrix},\qquad
 L=\im(\Sp(U)\times\Sp(W)).
\]
Every automorphism of $\mathfrak{sp}(U\oplus W)$ is inner. A
symplectic representative centralizing the involution commutes or
anticommutes with $\diag(I_U,-I_W)$. Commuting representatives are
block diagonal. Anticommuting representatives exist exactly when
$p=q$; they are products of block diagonal matrices with
$\left(\begin{smallmatrix}0&I\\I&0\end{smallmatrix}\right)$, whose
action is $C\mapsto-C^*$. Hence $A=L$ for $p\ne q$, and
$A=\langle L,s\rangle$ for $p=q$. Proposition~\ref{prop:symplectic-linear}
settles the latter case.

\emph{BDI.} Let $U,W$ be quadratic of dimensions $p<q$.
The BDI row of \cite[Table 2]{HY} identifies the full effective
automorphism group with the action induced by
$\Or(U)\times\Or(W)$, including the triality cases in total
dimension eight. Proposition~\ref{prop:orthogonal-linear} therefore
applies for every $p<q$. The following matrix argument also explains
this description. Away from total dimension eight, every automorphism of
$\mathfrak{so}(U\oplus W)$ is induced by an orthogonal matrix. A
representative centralizing the involution preserves or exchanges the
eigenspaces of $\diag(I_U,-I_W)$. Their unequal dimensions exclude
exchange. Thus its action lies in $\Or(U)\times\Or(W)$, and
Proposition~\ref{prop:orthogonal-linear} applies.

In total dimension eight, the possibilities are $(1,7)$, $(2,6)$,
and $(3,5)$, up to reversal. For $(3,5)$, as for $(1,7)$, the
outer class of $\theta$ is the transposition fixing the vector
module, so its centralizer in the triality group fixes that module.
For completeness, we also check $(1,7)$ and $(2,6)$ directly.
An automorphism of either pair preserves
the eight dimensional vector representation. For $(1,7)$ the outer
class of $\theta$ is the transposition in
$\operatorname{Out}(\mathfrak{so}_8)=S_3$ fixing the vector module.
Its centralizer in $S_3$ is the same order two subgroup. For $(2,6)$,
the vector module restricts to the fixed algebra with irreducible
dimensions $1,1,6$, whereas each half spin module has dimensions $4,4$.
An automorphism preserving the fixed algebra preserves these multisets
and hence cannot exchange the vector module with a half spin module.
In either case an intertwiner on the vector module preserves its
unique invariant symmetric form up to scalar. Rescaling the intertwiner
realizes the automorphism by orthogonal conjugation. The preceding
argument applies.

\emph{DIII.} Use the hyperbolic space $V\oplus V^*$, with
$\dim V=n$. The involution is represented by
$d=\diag(\ii I,-\ii I)$, and
$X=\Lambda^2V\oplus\Lambda^2V^*$. When $n\ne4$, every automorphism
is orthogonally induced. Representatives commuting or anticommuting
with $d$ have the forms
\[
 \begin{pmatrix}b&0\\0&b^{-t}\end{pmatrix},\qquad
 \begin{pmatrix}0&b\\b^{-t}&0\end{pmatrix},
\]
so $A=\langle L,s\rangle$. Both parity choices for $n$ are allowed in
the full orthogonal group. Proposition~\ref{prop:paired-linear} applies.

For $n=4$, triality identifies the pair with BDI of dimensions $(2,6)$.
Indeed, one half spin module restricts to $\mathfrak{gl}_4$ as
$\Lambda^0V\oplus\Lambda^2V\oplus\Lambda^4V$, with its determinant
twist. The first and last lines pair nondegenerately, and the middle
summand is a nondegenerate quadratic space of dimension six. The image
of $\mathfrak{gl}_4$ is therefore
$\mathfrak{so}_2\oplus\mathfrak{so}_6$: the representation is faithful
and both algebras have dimension sixteen. Choosing this half spin
module as the vector module is a triality automorphism. It identifies
the involutions as well, since the minus space is the orthogonal
complement of the fixed algebra for the Killing form. It therefore conjugates the full
automorphism groups of the pairs. Transporting the BDI result completes
the proof.
\end{proof}

\section{Homogeneity and normal symbols}\label{sec:homogeneity}

We now study distributions whose support and Fourier support lie in
the nilpotent cone. Under the stability assumptions below, two commuting
$\mathfrak{sl}_2$ actions produce a nonzero homogeneous distribution
with one of four dilation characters. Comparison with normal weights
proves vanishing for the nice blocks and restricts the possible support
in EIII and EVII.

Let $V$ be a complex vector space of dimension $n>0$, with a nondegenerate
symmetric bilinear form $B$. Put $q(x)=B(x,x)$ and use the dilation action
\[
 (\rho(t)T)(\varphi)=T\bigl(\varphi(t^{-1}\,\cdot)\bigr),
 \qquad t\in\C^\times.
\]
Fourier transformation uses the kernel
$\exp(2\pi\ii\operatorname{Re}B(x,y))$ and the self dual measure.

\begin{lemma}\label{lem:sl2-finite}
Suppose $e,f,h$ satisfy the $\mathfrak{sl}_2$ relations on a complex
vector space. If $e^kv=f^mv=0$ for some positive integers $k,m$, then
$U(\mathfrak{sl}_2)v$ is finite dimensional.
\end{lemma}
\begin{proof}
First suppose $ev=0$. The commutation relations give
\[
 e^mf^mv=m!\prod_{j=0}^{m-1}(h-j)v.
\]
The left side is zero. Hence $\C[h]v$ is finite dimensional; it is
killed by $e$ and $f^m$, as follows by commuting these operators with
polynomials in $h$. The Poincar\'e--Birkhoff--Witt theorem therefore gives
\[
 U(\mathfrak{sl}_2)v=\sum_{j=0}^{m-1}f^j\C[h]v,
\]
which is finite dimensional.

We now argue by induction on $k$. For $k>1$, put $w=e^{k-1}v$.
In PBW order, every monomial $e^ah^bf^c$ occurring in
$f^{m+k-1}e^{k-1}$ has $c-a=m$, and thus $c\geq m$.
It follows that $f^{m+k-1}w=0$. Since $ew=0$, the first case makes
$Uw$ finite dimensional. The image of $v$ in $Uv/Uw$ is killed by
$e^{k-1}$ and $f^m$. The induction hypothesis makes this quotient
finite dimensional as well, proving the assertion.
\end{proof}

The next lemma is a complex form of the Archimedean homogeneity
argument of Aizenbud and Gourevitch
\cite[Theorem 5.1.7 and Appendix C]{AG}. We state the consequence yielding the four characters needed here and give its proof using the commuting
holomorphic and antiholomorphic $\mathfrak{sl}_2$ actions; this is
the precise formulation used in the later orbit calculations.

\begin{lemma}\label{lem:four-character}
Let $0\ne M\subset\cS^*(V)$ be a complex vector space of distributions
supported on $\{q=0\}$. Suppose that $M$ is stable under $\cF_B$ and
multiplication by $q$ and $\bar q$. Then there are $0\ne T\in M$ and
$a,b\in\{0,1\}$ such that
\begin{equation}\label{eq:four-character}
 \rho(t)T=|t|_{\C}^{n/2}t^a\bar t^bT,
 \qquad t\in\C^\times.
\end{equation}
\end{lemma}
\begin{proof}
We shall obtain the required characters from two commuting
$\mathfrak{sl}_2$ actions. In complex linear coordinates, put
\[
 \Delta_B=\sum_{i,j}(B^{-1})_{ij}\partial_{z_i}\partial_{z_j},
 \qquad E_z=\sum_i z_i\partial_{z_i}.
\]
The identity $[\Delta_B,q]=4E_z+2n$ gives an $\mathfrak{sl}_2$ triple
\[
 e=q/2,\qquad f=-\Delta_B/2,\qquad h=E_z+n/2.
\]
Its complex conjugate is a second triple commuting with the first.
The Fourier identities
\[
 \cF_B(qT)=-\pi^{-2}\Delta_B\cF_BT,
 \qquad
 \cF_B(\bar qT)=-\pi^{-2}\overline\Delta_B\cF_BT
\]
show that the four root operators
\[
 e=q/2,\quad f=-\Delta_B/2,\quad
 \bar e=\bar q/2,\quad \bar f=-\overline\Delta_B/2
\]
preserve $M$.

Each root operator is locally nilpotent on $M$. Indeed, a tempered
distribution has a uniform finite differential order. If its support
lies in $\{q=0\}$, a sufficiently high power of $q$, or of $\bar q$,
vanishes to more than that order along its support and annihilates it.
Applying the same observation to its Fourier transform proves local
nilpotence of the two Laplacians. Lemma~\ref{lem:sl2-finite} consequently
makes each cyclic module for either triple finite dimensional. Since the
triples commute, their joint cyclic modules are finite dimensional: apply
the second triple to a finite basis of the cyclic module for the first.

A nonzero joint cyclic module contains a simple submodule
$V_r\otimes V_s$, where $V_j$ is the irreducible
$\mathfrak{sl}_2$-module of highest weight $j$. The weights of $V_j$
include $0$ if $j$ is even and $1$ if $j$ is odd. Choose a nonzero joint
weight vector $T$ with weights $a,b\in\{0,1\}$.
On ordinary distributions, the generators of positive real dilations
and circle rotations are respectively
\[
 E_z+E_{\bar z}+2n=h+\bar h+n,
 \qquad \ii(E_z-E_{\bar z})=\ii(h-\bar h).
\]
Thus $\rho(r)T=r^{n+a+b}T$ for $r>0$ and
$\rho(e^{\ii\alpha})T=e^{\ii(a-b)\alpha}T$.
These two identities give \eqref{eq:four-character}.
\end{proof}

We apply this lemma to the isotropy space $\fp$ of a complex semisimple
symmetric pair of Lie algebras $(\fg,\fh)$, with eigenspace decomposition
\[
 \fg=\fh\oplus\fp.
\]
Let $L$ denote the connected effective isotropy group, and let
$\cN\subset\fp$ be the nilpotent cone. Let $B$ be the Killing form
of $\fg$, and use the same notation for its restriction to $\fp$.
Since $B$ is $\theta$-invariant, $\fh$ and $\fp$ are orthogonal,
and its restrictions to both spaces are nondegenerate.
Set $n=\dim_{\C}\fp$.

\subsection{Normal triples and normal symbols}
The existence and conjugacy of normal triples belong to the
Kostant--Rallis theory \cite{KR}. The first assertion below is the
conjugacy statement with $e$ fixed. We include an explicit proof to
record the unipotent conjugating element and the consequences for
the full stabilizer, its determinant, and the central pairing.

\begin{lemma}\label{lem:triple-conjugacy}
Two normal triples with the same element $e$ are conjugate by a
unipotent element of $L_e^\circ$. Every element of $L_e$ is a
product of such an element and an element of the full triple
centralizer. The unipotent factor has determinant one on $\fh^e$.
If $Z\in\fz(\fh)$ and $B$ is an invariant bilinear form, then
$B(Z,d)$ is independent of the normal triple containing $e$.
\end{lemma}
\begin{proof}
Let $(e,d,f)$ and $(e,d',f')$ be normal triples. Their relations give
$d'-d\in\fh^e\cap[e,\fp]$. In the decomposition under $(e,d,f)$,
the space $\fh^e$ consists of highest vectors with nonnegative
$d$-weights. Its intersection with $[e,\fp]$ has only positive
weights, since the highest vectors of weight zero belong to
trivial summands.

Let $z_j$ be the component of least weight $j>0$ in $d'-d$.
Conjugation by $u_j=\exp(z_j/j)\in L_e^\circ$ removes this component:
$[z_j/j,d]=-z_j$, and all other new terms have larger weight.
Iteration terminates, since the weights of $\fg$ are bounded.
The product $u$ lies in the unipotent subgroup defined by positive
weights and satisfies $\Ad(u)d'=d$. Then
$f''=\Ad(u)f'$ satisfies
$[e,f''-f]=0$ and $[d,f''-f]=-2(f''-f)$.
There are no negative weights in $\fp^e$, so $f''=f$.
All these conjugations act with determinant one on $\fh^e$.
Apply this construction to the triple obtained from an element
of $L_e$ to obtain the asserted factorization. Finally $L$ fixes
$Z$, so invariance of $B$ gives the last assertion.
\end{proof}

For a normal triple $(e,d,f)$, the integral $d$-weight decomposition
defines automorphisms of $\fg$ by
\[
 v\longmapsto t^jv\qquad\text{if }[d,v]=jv.
\]
These automorphisms preserve brackets and commute with $\theta$.
Their connected image therefore lies in
$\Int(\fg)^{\theta,\circ}$ and gives a cocharacter $D_t\in L$.
Thus $\Ad(D_t)v=t^jv$ on the $d$-weight-$j$ space. Put
\[
 \tau(e)=\tr\bigl(\ad(d)|_{\fh^e}\bigr),\qquad d_e=\tau(e)-n.
\]
The operator $A_t=t^{-2}\Ad(D_t)$ fixes $e$. The normal quotient to
$Le$ is identified, by the triple decomposition, with
\[
 \fp/[\fh,e]\simeq\fp^f.
\]
A lowest normal line of $d$-weight $-\ell$ has $A_t$-weight
$t^{-\ell-2}$. We call $\ell+2$ its normal degree.
The inverse action on the affine slice $e+\fp^f$ is the
\emph{Kazhdan dilation}
\begin{equation}\label{eq:kazhdan}
 \kappa_t(e+n)=t^2\Ad(D_t^{-1})(e+n)
 =e+\sum_{\ell\geq0}t^{\ell+2}n_{-\ell},
 \qquad n=\sum_{\ell\geq0}n_{-\ell},\quad[d,n_{-\ell}]=-\ell n_{-\ell}.
\end{equation}
Thus $e$ is fixed and all slice coordinates have positive weights;
in particular, $\kappa_t$ contracts the slice to $e$ as $t\to0$.

We use the smooth wave front set
$\WF(T)\subset T^*(\fp_\R)\setminus\{\text{zero covectors}\}$;
see \cite[Section 8.1]{Hormander}. In local real coordinates,
$(x_0,\xi_0)\notin\WF(T)$ means that there are
$\chi\in C_c^\infty$, equal to one near $x_0$, and an open conic
neighborhood $\Gamma$ of $\xi_0$ such that, for every $N\geq0$,
\[
 |\widehat{\chi T}(\xi)|\leq C_N(1+\|\xi\|)^{-N}
 \qquad(\xi\in\Gamma).
\]
It records the points and cotangent directions in which a localized
distribution fails to have rapid Fourier decay. For an orbit $O=Le$,
$N_e^*O$ is its real conormal space, namely the covectors annihilating
$T_eO=[\fh,e]$.

\begin{lemma}\label{lem:normal-symbol}
Let $O=Le$ be open in the support of an $L$-invariant ordinary
distribution $T$ on $\fp$. Its highest normal symbol is a nonzero
smooth invariant section of
\[
 \Sym^m(N_{O,\R}\otimes_\R\C)\otimes\operatorname{Dens}_O
\]
for some $m\geq0$. Moreover,
$N_e^*O\setminus\{0\}\subseteq\WF(T)$.
\end{lemma}
\begin{proof}
Work in an invariant open neighborhood in which the support is
$O=Le$. In local real coordinates $(y,z)$ with $O=\{z=0\}$,
after restricting to a relatively compact coordinate neighborhood,
the structure theorem for distributions supported on a submanifold
\cite[Theorem 2.3.5]{Hormander} gives an expression
\[
 T(\varphi)=
 \sum_{|\alpha|\leq m}(-1)^{|\alpha|}
 u_\alpha\bigl(\partial_z^\alpha\varphi(y,0)\bigr),
\]
where the $u_\alpha$ are distributions in the $y$ variables.
Indeed, finite distribution order makes $T$ vanish on test functions
whose sufficiently high normal jet vanishes on $O$, and Taylor
expansion in $z$ gives the expression. The least possible normal
order is locally finite and is constant along $O$ by $L$-invariance.

The chain rule shows that the terms of highest normal order $m$
define a distributional section $\sigma_m(T)$ of
\[
 \Sym^m(N_{O,\R}\otimes_{\R}\C)
 \otimes \operatorname{Dens}_O,
\]
where $N_{O,\R}$ is the real normal bundle and
$\operatorname{Dens}_O$ is the complexified density bundle.
This section is $L$-invariant and nonzero.
To see that it is smooth, average it against a compactly supported
smooth function on $L$ of Haar integral one. Such averaging smooths
distributional sections on the homogeneous space $O$: in local
trivializations its kernel is smooth, by integration along the
fibers of the submersive orbit maps. Invariance makes the average
equal to $\sigma_m(T)$. Thus $\sigma_m(T)$ is smooth, and
transitivity implies $\sigma_m(T)(e)\ne0$.

In coordinates with $O=\{z=0\}$, let
$P_e(\eta)$ be its nonzero homogeneous polynomial in the real
normal covariable. If $P_e(\eta)\ne0$, every localization that is
nonzero at $e$ has a nonzero highest normal coefficient in that
direction. After Fourier transformation in the tangential variables,
this coefficient is nonzero at some fixed tangential frequency.
At that frequency and normal frequency $r\eta$, the Fourier
transform is a polynomial in $r$ with nonzero leading coefficient.
It therefore cannot decrease rapidly in a conic neighborhood of
$(0,\eta)$. Hence $(e,\eta)\in\WF(T)$.
The directions with $P_e(\eta)\ne0$ are dense in the conormal
fiber, so closedness of $\WF(T)$ gives
\[
 N_e^*O\setminus\{0\}\subseteq\WF(T).
\]

\end{proof}

\begin{lemma}\label{lem:normal-degrees}
Suppose an $L$-invariant distribution $T$ on $\fp$ has character
\eqref{eq:four-character}, and that the nilpotent orbit $Le$ is open
in its support. Every nonzero weight component of its leading normal
symbol has holomorphic and antiholomorphic normal degrees
\begin{equation}\label{eq:normal-degrees}
 D_+=d_e-2a,\qquad D_-=d_e-2b.
\end{equation}
In particular, $\tau(e)<n$ excludes this orbit from being open in the
support. The origin cannot support a nonzero distribution with
character \eqref{eq:four-character}.
\end{lemma}
\begin{proof}
By Lemma~\ref{lem:normal-symbol}, the highest normal symbol is
smooth and nonzero at $e$. We compute its character there.
The map $\fh/\fh^e\longrightarrow T_e(Le)$ sending $A$ to $[A,e]$
intertwines $\Ad(D_t)$ on the quotient with $A_t$ on the tangent space.
The determinant of $\Ad(D_t)$ on the reductive Lie algebra $\fh$ is
one, whereas its determinant on $\fh^e$ is $t^{\tau(e)}$. Therefore
\[
 \det_{\C}(A_t|_{T_e(Le)})=t^{-\tau(e)}.
\]
Pushforward of an ordinary orbit density consequently contributes
$|t|_{\C}^{\tau(e)}$. A normal derivative of holomorphic degree $D_+$
and antiholomorphic degree $D_-$ contributes
$t^{-D_+}\bar t^{-D_-}$.

Since $T$ is $L$-invariant, its character under $A_t$ is its dilation
character at $t^2$, namely $|t|_{\C}^{n}t^{2a}\bar t^{2b}$.
Comparing the two expressions on any nonzero weight component gives
\[
 t^{-D_+}\bar t^{-D_-}|t|_{\C}^{\tau(e)}
 =|t|_{\C}^{n}t^{2a}\bar t^{2b},
\]
which is \eqref{eq:normal-degrees}. The normal degrees are nonnegative,
so $\tau(e)<n$ is impossible.

Finally, a distribution supported at the origin is a finite sum of jets
$\partial_z^\alpha\partial_{\bar z}^\beta\delta_0$. Its positive real
dilation exponents are $-|\alpha|-|\beta|$, while
\eqref{eq:four-character} has exponent $n+a+b>0$.
This proves the last assertion.
\end{proof}

The nilpotent cone contains finitely many $L$-orbits \cite{KR},
ordered by
inclusion of their closures. A maximal support orbit is open in the
support after removing the other orbit closures. We use such invariant
open neighborhoods in the normal symbol arguments below.

\begin{lemma}\label{evii:wavefront}
If $T\in\cS^*(\fp)^L$ satisfies
$\supp T,\supp\cF_BT\subset\cN$, every orbit $Le$ open in its
support is $\fp$-distinguished, that is,
\[
 \fp^e\subset\cN.
\]
\end{lemma}
\begin{proof}
Fourier support in a closed cone $C$ gives
\begin{equation}\label{eq:wavefront-cone}
 \WF(T)\subset\fp\times(C\setminus\{0\}).
\end{equation}
Indeed, for $\psi\in C_c^\infty(\fp)$, the convolution
$\widehat{\psi T}=\widehat\psi*\widehat T$ decreases rapidly in
closed cones separated from $C$.

Put $O=Le$. Lemma~\ref{lem:normal-symbol} gives
$N_e^*O\setminus\{0\}\subseteq\WF(T)$.

We identify this fiber using the form fixed above. For $v\in\fp$
and $A\in\fh$, invariance gives
\[
 B(v,[A,e])=B([e,v],A).
\]
Since $[e,v]\in\fh$ and $B|_{\fh}$ is nondegenerate,
\[
 [\fh,e]^{\perp_B}\cap\fp=\fp^e.
\]
The annihilator for $\operatorname{Re}B$ is the same real subspace:
vanishing against both $A$ and $\ii A$ is equivalent to vanishing
of the complex pairing. Thus
\[
 N_e^*O=\fp^e
\]
under the identification of real covectors with $\fp$ by
$\operatorname{Re}B$. Applying \eqref{eq:wavefront-cone}
with $C=\cN$ proves $\fp^e\subset\cN$.
\end{proof}

\subsection{Distributions on a transverse slice}
We record the density conversion used below. Let $L$ act linearly
on a real vector space $V$, and let $N$ be an affine slice such
that the action map $\alpha:L\times N\to V$ is a submersion.
Locally, the pullback of an invariant generalized function has
the form $1\otimes u$: invariance under left translations in the
first variable makes it constant in that variable. Fix constant
positive densities on $V$ and $N$.

\begin{lemma}\label{lem:slice-density}
Let $T$ be an $L$-invariant ordinary distribution on $V$,
and let $S$ be the ordinary distribution on $N$ obtained by
pulling back its associated generalized function.
Suppose affine maps $b$ on $V$ and $c$ on $N$, and an
automorphism $\psi$ of $L$, satisfy
\[
 b\alpha(h,n)=\alpha(\psi(h),c(n)),\qquad b_*T=\chi T.
\]
If their absolute Jacobians are the constants $J_V$ and $J_N$,
then
\begin{equation}\label{eq:slice-density}
 c_*S=\chi\,\frac{J_V}{J_N}S.
\end{equation}
The statement also holds for germs wherever the action map is
a submersion.
\end{lemma}
\begin{proof}
Writing an ordinary distribution as a generalized function times
the chosen density changes its character under $b$ from $\chi$
to $\chi J_V$. Pullback preserves this character by the displayed
equivariance. The constant generalized function on $L$ is fixed
by $\psi$, so $u$ has character $\chi J_V$ under $c$.
Multiplying by the density on $N$ divides the character by $J_N$,
which proves \eqref{eq:slice-density}.
In submersion coordinates the pullback is the tensor product
with the constant generalized function in the fibre variables.
This also proves the assertions about local pullback and germs.
\end{proof}

\subsection{Nice blocks}\label{sec:nice}
The vanishing result for nice pairs originates in Sekiguchi's
work on invariant spherical hyperfunctions \cite{Sekiguchi}.
The numerical input used below is precisely
\cite[Theorem 6.3]{Sekiguchi}. An earlier algebraic vanishing theorem
under this inequality is \cite[Theorem 3.8]{LS}, formulated for
equivariant $\DD$-modules. The formulation with simultaneous
nilpotent support and Fourier support follows from the proof of
\cite[Theorem 5.2.4 and Corollary 5.2.7]{Aiz}; we do not claim
this vanishing as a new result. We use the nice families listed in
\cite[Definition 5.2.5]{Aiz} and give a proof in our notation.
References to \cite{Aiz} use the numbering of its final published
version.

\begin{proposition}\label{prop:nice}
For every nontrivial nice semisimple symmetric block, $\cC=0$.
In particular, \eqref{eq:F} holds.
\end{proposition}
\begin{proof}
For a distinguished nilpotent element $e$ in a nice block,
Sekiguchi's numerical inequality \cite[Theorem 6.3]{Sekiguchi}, in the
formulation in terms of negative defect in
\cite[Definition 5.1.33, Theorem 5.2.6 and Corollary 5.2.7]{Aiz},
gives
\begin{equation}\label{eq:nice-defect}
 \tau(e)=\tr(\ad d|_{\fh^e})<\dim_\C\fp.
\end{equation}
In the grading $\fh\oplus\fp$, Aizenbud's defect is exactly
$\tau(e)-\dim_\C\fp$: the even part of $\ker\ad e$ is $\fh^e$.
Thus the cited negative defect assertion has precisely the form
\eqref{eq:nice-defect}.

Suppose $\cC\ne0$. First select a homogeneous distribution. The space $\cC$ is stable under Fourier
transformation and multiplication by the Killing quadratic form and
its conjugate: on the Fourier side these multiplications become
constant coefficient differential operators. Lemma~\ref{lem:four-character}
therefore gives a nonzero $T\in\cC$ with one of the four dilation
characters. Next choose an orbit open in its support; this is possible
because the nilpotent cone has finitely many $L$-orbits. The origin
is excluded by Lemma~\ref{lem:normal-degrees}. For a nonzero orbit,
Lemma~\ref{evii:wavefront} forces $e$ to be distinguished, so
\eqref{eq:nice-defect} applies. Finally the normal degree formula gives
\[
 D_+=\tau(e)-\dim_\C\fp-2a<0,
\]
although a normal symbol has nonnegative weighted degree. This
contradiction excludes every possible open support orbit, proving
$\cC=0$.
\end{proof}

\section{Automorphisms of exceptional blocks}\label{sec:exceptional}
We use the exceptional centralizers of involutions in
\cite[Section 5, Table 2]{HY} to prove \eqref{eq:F} for seven
exceptional blocks. The passage from their compact forms to complex
groups is the polar decomposition argument in
Section~\ref{sec:classical}. For EIII and EVII, the same calculation leaves at most
one coset beyond the connected isotropy group. Retain
$L=\Aut(\fg,\theta)^\circ|_\fp$ and
$A=\Aut(\fg,\theta)|_\fp$. Since the Lie algebra of the centralizer of
$\theta$ in $\Aut(\fg)$ is $\ad\fh$, one has $A^\circ=L$.

\begin{proposition}\label{prop:exceptional-aut}
For FII, EVI, EIX, EVIII, FI, and GI, $A=L$.
For EIV, $K_Y=L$. For EIII and EVII, $[A:L]=2$.
In particular, FII, EVI, EIX, EVIII, FI, GI, and EIV satisfy
\eqref{eq:F}.
\end{proposition}
\begin{proof}
Let $\widetilde G$ be the simply connected semisimple group with Lie
algebra $\fg$, and lift $\theta$ to it. An inner automorphism represented
by $a\in\widetilde G$ commutes with $\theta$ exactly when
\begin{equation}\label{eq:central-cocycle}
 \theta(a)=za,\qquad z\in Z(\widetilde G),\qquad z\theta(z)=1.
\end{equation}
The fixed subgroup $\widetilde G^\theta$ is connected
\cite[Theorem 8.1]{Steinberg}. Thus an element $a$ that can be modified
by a central element to lie in $\widetilde G^\theta$ acts through $L$.
We apply this observation to the centres of the exceptional groups.

For FII, FI, EIX, EVIII, and GI the simply connected groups have
trivial centre, and $\mathfrak f_4$, $\mathfrak e_8$, and
$\mathfrak g_2$ have no outer automorphisms.
Equation~\eqref{eq:central-cocycle} therefore gives $A=L$.

For EVI there are again no outer automorphisms. The centre of
$\widetilde E_7$ has order two. On the $56$-dimensional representation,
the branching under $\Spin_{12}\cdot\SL_2$ is
\[
 \mathbf{56}=(\mathbf{32},\mathbf1)\oplus(\mathbf{12},\mathbf2).
\]
One may represent the involution by the image $t$ of $(1,-I_2)$; its
two eigenspaces have dimensions $32$ and $24$. The nontrivial central
element acts by $-I_{56}$. The alternative $\theta(a)=-a$ in
\eqref{eq:central-cocycle} would conjugate $t$ to $-t$, which is
impossible because their traces are $8$ and $-8$. Consequently $A=L$.
This realization is also used in \cite[Theorem 4.20]{Rubio}.

For EIII and EIV the centre of $\widetilde E_6$ is $\mu_3$. In EIII
the involution acts trivially on it, so \eqref{eq:central-cocycle}
implies $z^2=1$ and hence $z=1$. In EIV the involution inverts the
centre. Replacing $a$ by $ca$, for $c\in\mu_3$, replaces $z$ by
$zc^{-2}=zc$, so it can again be made equal to one. The effective
inner centralizer is therefore $L$ in both cases. Since the outer
automorphism group of $\mathfrak e_6$ has order two, $[A:L]\leq2$.
For EIII, the additional involution in \cite[Table 2]{HY} realizes
the second coset, so equality holds. This group calculation is a
cited result; the distributional treatment of its extra coset is
given in Section~\ref{eiii:section}.

For EIV, $\theta$ itself is outer and acts as $-1$ on $\fp$.
Thus $A=\langle L,-1\rangle$. Identify $\fp$ with the traceless
subspace $J_0$ of the complex Albert algebra. Its cubic norm restricts
to a nonzero odd $F_4$-invariant polynomial: its value on the diagonal
element $(1,1,-2)$ is $-2$. The extra coset acts nontrivially on the
invariant quotient, and therefore $K_Y=L$.

For EVII write $\fh=\mathfrak e_6\oplus\C Z$, with
$\ad Z$ having the two opposite eigenvalues on
$\fp=\mathbf{27}\oplus\mathbf{27}^*$. Every automorphism of
$\mathfrak e_7$ is inner. An element of $A$ must carry $Z$ to $Z$
or $-Z$. The kernel is the centralizer of $Z$ in the adjoint group,
a connected Levi subgroup with Lie algebra $\fh$, whose effective
action is $L$. The explicit involution in
\eqref{evii:actual-swap} commutes with $\theta$ and sends $Z$ to $-Z$.
It provides the second coset. This proves all the assertions.
\end{proof}

\section{The outer action in type EIII}\label{eiii:section}
Let $L$ be the connected effective isotropy group of the symmetric
pair of Lie algebras
\[
 (\fg,\fh)=(\mathfrak e_6,\mathfrak{so}_{10}\oplus\C Z_0),
 \qquad \fp=S^+\oplus S^-.
\]
The two half spin modules $S^+$ and $S^-$ are dual, have dimension $16$,
and have $Z_0$-weights $1$ and $-1$.  Write $\cN$ for the nullcone and
\[
 \cC=\{T\in\cS^*(\fp)^L:
             \supp T\subseteq\cN,\quad\supp\cF_BT\subseteq\cN\},
\]
where $B$ is the restriction of the Killing form.  We prove that every
automorphism of the symmetric pair of Lie algebras which acts trivially on
$\fp\sslash L$ fixes $\cC$.

\begin{proposition}\label{eiii:vanishing}
Let $s$ be an automorphism of the EIII symmetric pair of Lie algebras whose action
on $\fp\sslash L$ is the identity.  If $s^2\in L$, then
\[
 \cC^{s=-1}=0.
\]
Consequently the kernel of the action of
$\Aut(\fg,\theta)|_{\fp}$ on $\fp\sslash L$ fixes $\cC$.
\end{proposition}

The effective inner centralizer of $\theta$ is $L$, and the outer
automorphism group of $\mathfrak e_6$ has order two.  Thus every
nonidentity coset relevant to the last assertion has a representative
as in the first assertion.  Its square acts trivially on $L$-invariant
distributions, so it suffices to exclude its odd eigenspace.

We shall exclude every orbit that could be open in the support of an
odd distribution. Homogeneity first leaves only four orbits. Two of
them admit no invariant orbital density; on the other two the
involution fixes every possible leading normal symbol. Since there
are finitely many nilpotent orbits, these exclusions prove the
proposition. We begin by explaining the orbit data used in this
argument.

\subsection{Nilpotent characteristics and normal degrees}
Assume that $\mathcal M=\cC^{s=-1}$ is nonzero.  The automorphism $s$
preserves the Killing form, so $\mathcal M$ is stable under Fourier
transformation and multiplication by $q=B(x,x)$ and $\bar q$.
For the latter assertion, Fourier transformation turns these
multiplications into constant coefficient differential operators, which
preserve the Fourier support condition.  Lemma~\ref{lem:four-character}
therefore gives $0\ne T\in\mathcal M$ satisfying
\begin{equation}\label{eiii:homogeneity}
 \rho(t)T=|t|_{\C}^{16}t^a\bar t^bT,
 \qquad a,b\in\{0,1\}.
\end{equation}

We first determine the orbits which can be open in the support of $T$.
For a normal triple $(e,d,f)$ put
\[
 \tau(e)=\tr(\ad d|_{\fh^e}),\qquad \delta_e=\tau(e)-32.
\]
By Lemma~\ref{lem:normal-degrees}, a nonzero leading normal symbol along
$Le$ has holomorphic and antiholomorphic weighted degrees
\begin{equation}\label{eiii:normal-degrees}
 D_+=\delta_e-2a,\qquad D_-=\delta_e-2b.
\end{equation}
Each elementary normal degree is positive.  In particular,
$\tau(e)<32$ excludes the orbit, while $\tau(e)=32$ allows only
$a=b=0$ and normal order zero.  The same lemma excludes the origin.

The nonzero nilpotent orbits are numbered as in
\cite[Table~9]{DjokovicE6} and \cite[Table~14]{DS}.  The following table
gives the data needed in \eqref{eiii:normal-degrees}. For a fixed
normal triple, set
\[
 \begin{split}
 \fg(j)&=\{v\in\fg:[d,v]=jv\},\quad
 \fh(j)=\fh\cap\fg(j),\quad \fp(j)=\fp\cap\fg(j),\\
 g_j&=\dim\fg(j),\qquad h_j=\dim\fh(j),\qquad
 p_j=\dim\fp(j)=g_j-h_j\quad(j\in\Z).
 \end{split}
\]
For $j\geq0$ define the highest vector multiplicities
\[
 m_{j,\fh}=\dim(\fh^e\cap\fg(j)),\qquad
 m_{j,\fp}=\dim(\fp^e\cap\fg(j)),
\]
and their generating polynomials
\[
 m_{\fh}(t)=\sum_{j\geq0}m_{j,\fh}t^j,\qquad
 m_{\fp}(t)=\sum_{j\geq0}m_{j,\fp}t^j.
\]
Here $t$ is a formal variable. A triple summand means an irreducible
module for the $\mathfrak{sl}_2$ generated by $(e,d,f)$; a summand
of highest weight $j$ has dimension $j+1$. Thus $m_{j,\fh}$ and
$m_{j,\fp}$ count those summands according to the $\theta$-eigenspace
containing their highest vector. The orbit dimension is complex,
and the type in the second column is the nilpotent orbit label used
in the cited tables.
\begin{center}\small
\begin{tabular}{@{}ccllrr@{}}
\toprule
Orbit&Type&$m_{\fh}(t)$&$m_{\fp}(t)$&$\tau(e)$&$\dim Le$\\
\midrule
1,2&$A_1$&$25+10t$&$10+10t+t^2$&10&11\\
3,4&$2A_1$&$22+8t$&$8t+8t^2$&8&16\\
5&$2A_1$&$16+8t+6t^2$&$6+8t+2t^2$&20&16\\
6&$A_2$&$12+12t^2+t^4$&$4+7t^2$&28&21\\
7,8&$A_2+A_1$&$9+4t+6t^2+3t^3+t^4$&$4t+2t^2+3t^3$&29&23\\
9&$2A_2$&$14+8t^4$&$8t^2$&32&24\\
10,11&$A_3$&$11+4t^3+5t^4$&$t^2+4t^3+t^6$&32&26\\
12&$A_4$&$4+4t^2+3t^4+4t^6+t^8$&$t^2+t^6$&52&30\\
\bottomrule
\end{tabular}
\end{center}

Here is how to recover the table from root data. The
\emph{characteristic} of a nilpotent orbit is the semisimple element
$d$ of a normal triple, up to conjugacy. After conjugating $d$ into a
chosen Cartan subalgebra, its labels are the values of the simple
roots on $d$. A dominant characteristic has nonnegative simple root
labels. To compute the spectrum on $\fg$, choose the representative
dominant for the $E_6$ Weyl group. To compute the spectrum on $\fh$,
choose a representative dominant for the Weyl group of its $D_5$
factor. These are two conjugates of the same characteristic in
$E_6$, chosen for different computations.

The numbering is specified by the edges of the Dynkin diagrams:
for $E_6$ the adjacent node pairs are $(1,3),(3,4),(4,5),(5,6),(2,4)$;
for $D_5$ they are $(1,2),(2,3),(3,4),(3,5)$. A string such as
$010000$ denotes the ordered tuple $(0,1,0,0,0,0)$ of simple root
values. The required labels are
\[
\begin{array}{c|c|c}
\text{orbit}&E_6\text{ labels}&D_5\text{ labels}\\\hline
1,2&010000&00001\text{ or }00010\\
3,4&100001&10000\\
5&100001&00011\\
6&020000&02000\\
7,8&110001&11010\text{ or }11001\\
9&200002&40000\\
10,11&120001&00013\text{ or }00031\\
12&220002&02022.
\end{array}
\]
For a root $\alpha=\sum_i c_i\alpha_i$, its root space has
$d$-weight $\alpha(d)=\sum_i c_i\alpha_i(d)$. Count these values
over all positive and negative $E_6$ roots to obtain $g_j$, and over
the $D_5$ roots to obtain $h_j$. Add six zero weights in each case:
the Cartan dimension of $\fg$ is six, and $\fh$ has a Cartan of
dimension five together with its one dimensional centre. The central
coordinate of $d$ acts trivially on $\fh$, so the $D_5$ labels
suffice for this count. Since conjugacy preserves the spectrum on
$\fg$, the two choices of representative give compatible values
$p_j=g_j-h_j$.

For $j\geq0$, the raising operator $\ad e$ maps
$\fh(j)$ onto $\fp(j+2)$ and $\fp(j)$ onto $\fh(j+2)$.
This follows on each irreducible $\mathfrak{sl}_2$ summand, where
raising is surjective above the middle weight. Its kernels are the
highest vectors just defined. Consequently
\begin{equation}\label{eiii:root-count}
 m_{j,\fh}=h_j-p_{j+2},\qquad
 m_{j,\fp}=p_j-h_{j+2},\qquad
 \tau(e)=\sum_{j\ge0}j\,m_{j,\fh},\qquad
 \dim Le=46-\sum_{j\ge0}m_{j,\fh}.
\end{equation}
These formulas recover both generating polynomials and the last two
columns. Indeed, $\fh^e$ is the direct sum of its highest vector
spaces, and $\dim Le=\dim\fh-\dim\fh^e$, with $\dim\fh=46$.
In the following subsection we also write $\fh_j=\fh(j)$ and
$\fp_j=\fp(j)$ for the weight spaces; the unadorned symbols $h_j,p_j$
continue to denote their dimensions. For example, orbit $10$ has
$\dim\fh^e=11+4+5=20$, so its dimension is $46-20=26$, while
$\tau(e)=3\cdot4+4\cdot5=32$.

The normal degree test now leaves exactly orbits $9,10,11,12$.
For $9,10,11$ it forces normal order zero. We first exclude $10,11$
by their stabilizers, then treat $9,12$ by the action of $s$.

\subsection{The stabilizers of the two $A_3$ orbits}
On orbits $10$ and $11$, equation~\eqref{eiii:normal-degrees} permits
only an invariant orbital density. Such a density exists on $L/L_e$
only if the stabilizer is unimodular, since $L$ is reductive and hence
unimodular. We shall find $Z\in\fh^e$ with
$\tr(\ad Z|_{\fh^e})\ne0$. The real one parameter subgroup it
generates has a nontrivial absolute determinant on $\fh^e$, proving
that the stabilizer is not unimodular.

Use orthonormal coordinates $\varepsilon_1,\ldots,\varepsilon_5$ for
$D_5$, whose roots are $\pm\varepsilon_i\pm\varepsilon_j$.  The weights
of $S^+$, respectively $S^-$, are
$\frac12(\pm\varepsilon_1\pm\cdots\pm\varepsilon_5)$ with an even,
respectively odd, number of minus signs.  A Cartan element
$(x_1,\ldots,x_5;c)$ acts by adding $c$ on $S^+$ and $-c$ on $S^-$.

We need the central coordinate of the two characteristics.  In
\cite[Tables~8 and~14]{DS} the simple roots are ordered as
\[
 (\beta_1,\ldots,\beta_5;\beta_6)
   =(\alpha_1,\alpha_3,\alpha_4,\alpha_5,\alpha_2;\alpha_6).
\]
Take $\beta_i=\varepsilon_i-\varepsilon_{i+1}$ for $1\le i\le4$
and $\beta_5=\varepsilon_4+\varepsilon_5$.  The $D_5$ part of
$\beta_6$ is
$(-\varepsilon_1-\varepsilon_2-\varepsilon_3-\varepsilon_4+
\varepsilon_5)/2$, and its root space has central weight $1$.
Consequently
\[
 c=\beta_6(d)+\tfrac12(x_1+x_2+x_3+x_4-x_5).
\]
The labels $(00013;-2)$ and $(00031;-6)$ therefore give
\begin{equation}\label{eiii:A3-characteristics}
 d_{10}=(2,2,2,2,1;3/2),\qquad
 d_{11}=(2,2,2,2,-1;-3/2).
\end{equation}

For orbit $10$, set
\[
 Z=(0,0,0,0,1;-1/2).
\]
The degree two space $\fp_2$ has seven weights.  Six are the weights
of $S^+$ with two minus signs among the first four coordinates and
a plus sign in the fifth; the last is
$\frac12(\varepsilon_1+\varepsilon_2+\varepsilon_3+
\varepsilon_4-\varepsilon_5)$ in $S^-$.  All have $Z$-weight zero.
Thus $[Z,e]=0$ for every $e$ with this characteristic.  The roots of
$\fh_0$ are $\varepsilon_i-\varepsilon_j$, $1\le i\ne j\le4$,
so $Z$ is central in $\fh_0$ and acts trivially on $\fh_0^e$.

The row for orbit $10$ in the table gives
\[
 \fh^e=\fh_0^e\oplus\fh_3\oplus(\fh_4\cap\fh^e),
 \qquad \dim\fh_0^e=11,
 \quad\dim\fh_3=4,
 \quad\dim(\fh_4\cap\fh^e)=5.
\]
The four roots of $\fh_3$ are $\varepsilon_i+\varepsilon_5$,
$1\le i\le4$, and have $Z$-weight $1$.  The six roots of $\fh_4$
are $\varepsilon_i+\varepsilon_j$, $1\le i<j\le4$, and have
$Z$-weight zero.  Hence
\begin{equation}\label{eiii:modular-trace}
 \tr(\ad Z|_{\fh^e})=4.
\end{equation}
For orbit $11$, use
$Z=(0,0,0,0,-1;1/2)$.  The degree three roots become
$\varepsilon_i-\varepsilon_5$, again of $Z$-weight $1$, and the
same trace calculation applies.

Along the real one parameter subgroup $\exp(tZ)$, the absolute real
determinant on $\fh^e$ is $e^{8t}$ by
\eqref{eiii:modular-trace}.  Thus both stabilizers $L_e$ are
nonunimodular.  Since $L$ is unimodular, neither homogeneous space
$L/L_e$ carries a nonzero invariant density.  Invariant distributions
on a transitive smooth manifold are smooth invariant densities, so
neither orbit can support the required symbol of normal order zero.

\subsection{The $2A_2$ orbit and the regular orbit}
Orbit $9$ is the unique nilpotent $L$-orbit of type $2A_2$; equivalently,
it is the unique orbit of dimension $24$ in the table.  Thus $s$
preserves it.  Its normal order is zero by
\eqref{eiii:normal-degrees}.  If an invariant orbital density exists,
choose it positive and denote it by $\mu$.  There is a positive scalar
$c$ with $s_*\mu=c\mu$.  Since $s^2\in L$, one has $c^2=1$ and
hence $c=1$.  Therefore an odd leading symbol cannot occur on orbit $9$.

Orbit $12$ permits positive normal order, so invariance of an orbital
density alone is insufficient. We instead construct two invariant
polynomials whose differentials identify its entire normal space.
An automorphism fixing the invariant quotient will then fix every
normal direction and every symmetric power used in a leading symbol.
The remaining density factor is controlled by positivity, as on
orbit $9$.  Use the root numbering and
brackets of Appendix~\ref{app:E7}.  The roots with $a_7=0$, together
with $H_1,\ldots,H_6$, give a Lie algebra of type $E_6$.  On this
subalgebra take $\theta(X_\alpha)=(-1)^{a_6}X_\alpha$ and
$\theta(H_j)=H_j$.  Its fixed algebra is
$\mathfrak{so}_{10}\oplus\C$, and the following central element has
weights $1,-1$ on its two half spin summands:
\[
 Z=\tfrac23H_1+H_2+\tfrac43H_3+2H_4+\tfrac53H_5+\tfrac43H_6.
\]
Consider the explicit triple
\begin{align*}
 e={}&X_{-23}+X_{-24}+X_{-27}+X_{40}+X_{42}+X_{43},\\
 d={}&4H_1+6H_2+8H_3+10H_4+6H_5,\\
 f={}&-4X_{23}-3X_{24}-3X_{27}-3X_{-40}-4X_{-42}-3X_{-43}.
\end{align*}
Substitution in the root brackets gives
$[d,e]=2e$, $[d,f]=-2f$, and $[e,f]=d$, with $e,f\in\fp$.
The simple $E_6$ labels of $d$ are $(0,2,2,0,2,-6)$, which in the
ordering $(\beta_1,\ldots,\beta_6)$ above are $(02022;-6)$.
Thus this is a normal triple for orbit $12$.  Formula
\eqref{eiii:root-count} gives $\dim[\fh,e]=30$.

Let $B_0$ be the invariant form specified in Appendix~\ref{app:E7} and
put
\[
 q_{2k}(Y)=B_0\bigl(Z,(\ad Y)^{2k}Z\bigr),\qquad k=1,2.
\]
The connected group $L$ fixes $Z$, so both polynomials are
$L$-invariant.  To compute their differentials, write
$R(Y)=[Y,[Y,Z]]$.  Invariance of $B_0$ and the weights of $Z$ give
$q_2(Y)=B_0(Y,Y)$ and $q_4(Y)=B_0(R(Y),R(Y))$.
The root brackets yield
\begin{align*}
 R(e)&=-2X_{15}-4X_{16}-2X_{17}+2X_{20}-2X_{21},\\
 dR_e(X_{23})&=-2(H_2+H_4+H_5+H_6),\\
 dR_e(X_6)&=-2(X_{-16}+X_{-17}+X_{-21}).
\end{align*}
Since $B_0(X_r,X_{-r})=-1$ and the Cartan and root spaces are
orthogonal, these identities give
\begin{equation}\label{eiii:quotient-jacobian}
 \begin{pmatrix}
 dq_2|_e(X_{23})&dq_2|_e(X_6)\\
 dq_4|_e(X_{23})&dq_4|_e(X_6)
 \end{pmatrix}
 =\begin{pmatrix}-2&0\\0&-32\end{pmatrix}.
\end{equation}
In particular, $d(q_2,q_4)_e$ has rank two.  Its kernel contains
$[\fh,e]$ by invariance and has the same dimension $30$.
Consequently it induces an isomorphism
\begin{equation}\label{eiii:normal-quotient}
 \fp/[\fh,e]\ \xrightarrow{\ \sim\ }\ \C^2.
\end{equation}
Every automorphism fixing $e$ and fixing $q_2,q_4$ acts trivially on
this normal space.

The orbit is the unique one of dimension $30$, so it is preserved by
$s$.  Choose $l\in L$ with $se=le$ and put $v=l^{-1}s$.  Then $ve=e$,
and $v^2\in L_e$ because $s$ normalizes $L$ and $s^2\in L$.
Every element of $L_e$ fixes $q_2,q_4$.  Since $s$ acts trivially on
$\fp\sslash L$, it fixes both polynomials, as does $v$.
Equation~\eqref{eiii:normal-quotient} therefore shows that $L_e$ and
$v$ act trivially on the normal space and all its real symmetric powers.

Consider the highest nonzero normal symbol of an invariant distribution
along this orbit.  Its fiber lies in a real symmetric power of
$\fp/[\fh,e]$ tensored with the orbit density line.  As $L_e$ acts
trivially on the first factor, a nonzero invariant symbol requires an
$L_e$-fixed density.  On this line $v$ acts by a positive scalar $c$;
the relation $v^2\in L_e$ gives $c^2=1$, hence $c=1$.
Thus $v$, and consequently $s$ on invariant sections, fixes every
possible leading symbol.  This excludes an odd distribution at every
finite normal order on $O_{12}$.

\begin{proof}[Proof of Proposition~\ref{eiii:vanishing}]
If $\cC^{s=-1}$ were nonzero, the preceding argument would produce a
nonzero homogeneous $T$ as in \eqref{eiii:homogeneity}.  Choose an orbit open in the support of $T$.  The normal degree calculation excludes the origin and
orbits $1$--$8$; nonunimodularity excludes $10$ and $11$; and the
actions on normal symbols exclude $9$ and $12$.  No such maximal
orbit exists, a contradiction.  Hence $\cC^{s=-1}=0$.
\end{proof}

\section{The EVII block}\label{sec:EVII}
In this section
\[
 \fh=\mathfrak e_6\oplus\C Z,\qquad
 \fp=\mathbf{27}\oplus\mathbf{27}^*,\qquad \dim_\C\fp=54.
\]
Let $L$ denote the connected effective isotropy group.
Proposition~\ref{prop:exceptional-aut} shows that the effective
automorphism group of the symmetric pair of Lie algebras is generated by
$L$ and an exchange $s$. We prove the required invariance on the space $\cC$
of \eqref{eq:C}.

\begin{proposition}\label{prop:evii}
For the EVII block, $\cC^{s=-1}=0$. Consequently \eqref{eq:F} holds.
\end{proposition}

The proof tests whether an odd distribution on a nilpotent orbit can
extend across its boundary. We use the orbit numbering of
\cite[Table 6]{DjokovicE7}. The first slice, at orbit $20$, meets
the nilpotent cone in two crossing axes. A torus acts with reciprocal
weights on the axes, and $s$ exchanges them. The logarithmic residues
of invariant distributions on the punctured axes force equal branch
coefficients. Oddness then forces them to vanish. This removes orbits
$20$, $21$, and $22$.

Homogeneity and the conormal condition leave only orbits $17$ and $18$
as possible open support orbits. Their common boundary contains orbit
$15$. The second slice, at that orbit, again meets the relevant support
in two crossing axes, but there are additional normal directions.
Homogeneity leaves three holomorphic and three antiholomorphic normal
terms, hence nine coefficients. The residue condition eliminates four
coefficients; invariance under one stabilizer element eliminates the
other five.

We explain the analytic mechanism before giving the root calculations.
For a normal triple $(e,d,f)$, set $N=\fp^f$. The decomposition
$\fp=[\fh,e]\oplus N$ makes
\[
 \alpha:L\times(e+N)\longrightarrow\fp,\qquad
 \alpha(l,e+n)=\Ad(l)(e+n),
\]
a submersion near $(1,e)$. An invariant generalized function pulls
back to $1\otimes u$, so the local problem reduces to a distribution
on the slice. Its support lies in the intersection of the original
support with $e+N$. Stabilizer actions and the Kazhdan dilation
\eqref{eq:kazhdan} constrain its coefficients. Lemma~\ref{lem:slice-density}
supplies the Jacobian factors when we pass back to ordinary
distributions. Finally, vanishing on the slice implies vanishing on
the open image of the action submersion. Thus a local obstruction at
a boundary point rules out an invariant distribution on the adjacent
orbits. Appendix~\ref{app:E7} gives the root conventions and exact
linear systems used below.

For a normal triple $(e,d,f)$ put
$\tau(e)=\tr(\ad d|_{\fh^e})$ and $d_e=\tau(e)-54$.
The complete trace calculation in Appendix~\ref{app:E7} gives exactly
the following nonzero orbits with $d_e\geq0$:
\begin{equation}\label{eq:EVII-six}
\begin{array}{c|c|r|r}
\text{orbit}&\text{lowest normal multiplicities}&\tau(e)&d_e\\\hline
15&2V_0+8V_2+2V_4&60&6\\
17,18&10V_2+V_6&60&6\\
20&V_2+2V_4+V_6&84&30\\
21,22&V_2+V_6+V_{10}&84&30
\end{array}
\end{equation}
Here $V_j$ denotes a triple summand of highest weight $j$ whose lowest
line lies in $\fp^f$; the corresponding normal degree is $j+2$.
Every other nonzero orbit has $\tau(e)<54$. By
Lemma~\ref{lem:normal-degrees}, such an orbit cannot be open in the
support of a distribution with one of the four dilation characters.

Orbit $15$ is not distinguished: its triple centralizer has a
two dimensional minus space $\fp^{e,d,f}$ on which the Killing form
is nondegenerate, so this space contains a nonnilpotent vector.
Lemma~\ref{evii:wavefront} excludes orbit $15$ from being open in the
support. It may still lie in the boundary of that support. This is
why a slice at orbit $15$ remains necessary later in the proof.

\begin{lemma}\label{lem:EVII-unimodular}
The full stabilizers $L_e$ of orbits $17$ and $18$ are unimodular.
\end{lemma}
\begin{proof}
By Lemma~\ref{lem:triple-conjugacy}, an element of $L_e$ is a
product of an element acting unipotently on $\fh^e$ and an
element of the full triple centralizer. It remains to compute
the determinant of the latter.

For orbits $17$ and $18$, the highest vectors lying in $\fh$
come from $21V_0+15V_4$. Put
\[
 E_j=\{v\in\fh^e:[d,v]=jv\},\qquad j=0,4.
\]
Thus $\fh^e=E_0\oplus E_4$.
The $\mathfrak{sl}_2$ decomposition and the nondegenerate
invariant form give nondegenerate forms
\[
 B_j(v,w)=B\bigl(v,(\ad f)^j w\bigr)
 \quad\text{on }E_j.
\]
These forms are symmetric because $j$ is even, and they are
preserved by the full triple centralizer. Its elements therefore
have determinant $\pm1$ on each $E_j$.

It follows that every $g\in L_e$ has absolute real determinant
one on $\fh^e$, so $L_e$ is unimodular. Moreover, $L$ has
determinant one on $\fh$; hence $L_e$ has absolute determinant
one on $\fh/\fh^e$. The orbit density line is therefore
$L_e$-trivial.
\end{proof}

\subsection{A distribution lemma at a node}
\begin{lemma}[Invariant distributions at a node]\label{evii:node}
Let $W$ be a complex vector space.  Let $\C^\times$ act on
$\C^2\oplus W$ by
\[
 a(u,v,w)=(au,a^{-1}v,w),
\]
and let $s(u,v,w)=(v,u,w)$.  A torus invariant ordinary distribution germ
at zero supported on
\[
 \{uv=0,\ w=0\}
\]
is $s$-invariant.  Hence its $s$-odd part is zero.
\end{lemma}
\begin{proof}
The issue is extension at the crossing point: invariance on the
punctured branches by itself imposes no relation between their
coefficients. The relation comes from the logarithmic term in such
an extension.

Put $d^\times z=dA(z)/(\pi|z|^2)$, where $dA$ is Lebesgue area
measure on $\C$. The required finite part is
\[
 F(\varphi)=\lim_{\delta\downarrow0}
 \left(\int_{|z|>\delta}\varphi(z)d^\times z
                  +2\log\delta\,\varphi(0)\right),
 \qquad\varphi\in C_c^\infty(\C).
\]
Taylor expansion shows convergence.  With $a_*F(\varphi)=F(\varphi(a\cdot))$,
change of variables gives
\begin{equation}\label{evii:finitepart}
 a_*F-F=-2\log|a|\,\delta_0.
\end{equation}

A distribution has bounded order on a sufficiently small compact
coordinate neighborhood.  On the punctured $u$-axis expand it in the
normal derivatives
$\partial_v^p\partial_{\bar v}^q\partial_w^\gamma\delta_{v=w=0}$;
$\gamma$ includes all real, or equivalently holomorphic and
antiholomorphic, derivatives in $W$.  Torus invariance determines each
coefficient on the punctured axis up to a constant: it is
$u^{-p}\bar u^{-q}d^\times u$.
The bounded distribution order leaves only finitely many indices.

Set $c_{pq}=(-1)^{p+q}/(p!q!)$.  An extension of such a branch term is
\[
 A_{pq\gamma}=c_{pq}
 (\partial_u^p\partial_{\bar u}^q F(u))\,
 \partial_v^p\partial_{\bar v}^q\delta(v)\,
 \partial_w^\gamma\delta(w).
\]
Let $B_{pq\gamma}=s_*A_{pq\gamma}$.  Differentiating
\eqref{evii:finitepart}, and including the reciprocal normal weight, gives
\[
 (a_*-1)A_{pq\gamma}=-2c_{pq}\log|a|\,J_{pq\gamma},\qquad
 (a_*-1)B_{pq\gamma}=+2c_{pq}\log|a|\,J_{pq\gamma},
\]
where
\[
 J_{pq\gamma}=
 \partial_u^p\partial_{\bar u}^q
 \partial_v^p\partial_{\bar v}^q
 \partial_w^\gamma\delta_0.
\]
These point jets are linearly independent and fixed by the torus.

Subtract the corresponding finite sums of $A$'s and $B$'s from the given
germ.  The difference is a finite point jet.  The finite dimensional
space of point jets splits into algebraic torus character spaces.
The projection of $(a_*-1)$ of any point jet onto the trivial character
space is zero.  Projecting the invariance equation and comparing the
independent $J_{pq\gamma}$ therefore forces equal coefficients for each
pair of branches.  Their sum is $s$-invariant.

The remaining odd part is a point jet fixed by the torus. Such a jet
has equal $u$- and $v$-derivative counts separately in the holomorphic and
antiholomorphic variables.  Since $s$ fixes $W$, it fixes every such jet.
Its odd part is zero as well.
\end{proof}

\subsection{The slice at orbit $20$}
We exclude orbits $20$, $21$, and $22$ by identifying the nilpotent
intersection of a transverse slice with a node. Once the centralizer
torus and the inner exchange have the actions of Lemma~\ref{evii:node},
that lemma eliminates every odd slice distribution.

Use the Chevalley basis of Appendix~\ref{app:E7}, with
$[X_r,X_{-r}]=-H_r$.  Put $S_r=X_r+X_{-r}$.  The Cartan data
for orbit 20 are
\begin{align*}
 H_0&=8H_1+10H_2+14H_3+20H_4+16H_5+12H_6+6H_7,\\
 E&=2X_1+\sqrt3X_6+\sqrt3X_{19}+2X_{37}
                         +\sqrt3X_{40}+\sqrt3X_{41},
\end{align*}
with $F$ obtained by reversing root indices.  The normal triple is
\[
 e=(E-F+\ii H_0)/2,\qquad d=\ii(E+F),\qquad f=(E-F-\ii H_0)/2.
\]
Two relevant central elements are
\[
 Z=S_7-S_{49}+S_{63}\in Z(\fh),\qquad
 z=3S_7-2S_5+S_3-S_{28}\in Z(\fh^{e,d,f}).
\]
For a root $\alpha=\sum a_j\alpha_j$, define the toral inner involution
\begin{equation}\label{evii:actual-swap}
 s(H_j)=H_j,\qquad
 s(X_\alpha)=(-1)^{a_2+a_3+a_5+a_7}X_\alpha.
\end{equation}
It is a Lie algebra automorphism because the signs form a character of
the root lattice.  Checking the simple root vectors in the explicit
$\theta$ table proves $s\theta=\theta s$; it then holds on all of $\fg$.
Substitution gives
\[
 se=e,\quad sd=d,\quad sf=f,\quad sZ=-Z,\quad sz=-z.
\]
Thus $s$ represents the nonidentity coset of the effective inner normalizer
modulo $L$.

The kernel $N=\fp^f$ has four dimensions.  Its $d$-weights are
$-2,-4,-4,-6$.  On the two dimensional weight $-4$ space, $\ad z$ has
characteristic polynomial $t^2+36$; it vanishes on the other two lines.
Choose $U$ with $[z,U]=-6\ii U$, put $V=sU$, and choose $W_8$ on the weight
$-6$ line.  The weight $-2$ line is spanned by $f$.
The involution fixes $f,W_8$ and exchanges $U,V$.  Consequently the affine
slice has coordinates
\begin{equation}\label{evii:slice}
 Y=e+uU+vV+w_4f+w_8W_8.
\end{equation}
The centralizer torus acts with reciprocal nontrivial weights on $u,v$
and trivially on $w_4,w_8$.  Its image on $u$ is all of $\C^\times$.
The Kazhdan dilation $\kappa_t(Y)=t^2\Ad(D_t^{-1})Y$,
defined in \eqref{eq:kazhdan}, fixes $e$. A normal vector of
$d$-weight $-j$ is multiplied by $t^{j+2}$, so the slice coordinates
have degrees
\[
 \deg(u,v,w_4,w_8)=(6,6,4,8).
\]

We determine the slice intersection in two steps. Three invariant
polynomials first force it into the two coordinate axes. We then
verify that both axes are nilpotent and identify the orbits they meet.
Let $B_0$ be the invariant form normalized by
\[
 B_0(H_i,H_j)=A_{ij},\qquad B_0(X_r,X_{-r})=-1,
\]
and set
\[
 q_{2k}(Y)=B_0(Z,(\ad Y)^{2k}Z),\qquad k=1,2,3.
\]
These are $L$-invariant polynomials.  They vanish on the nullcone,
because they have positive degree.  The centralizer torus weights and
Kazhdan degrees force their restrictions to the slice to have the forms
\begin{align}
 q_2&=a w_4,\label{evii:node-invariants}\\
 q_4&=b w_8+b' w_4^2,\notag\\
 q_6&=c uv+c'w_4w_8+c''w_4^3.\notag
\end{align}
The exact bracket calculation in Appendix~\ref{app:E7} gives, for its specified
normalizations,
\[
 a=-160,\qquad b=-12288,\qquad c=-1179648.
\]
Sequentially setting the
three polynomials to zero proves
\begin{equation}\label{evii:nil-contained-node}
 \cN\cap(e+\fp^f)\ \subset\ \{w_4=w_8=0,\ uv=0\}.
\end{equation}

Both punctured axes are nilpotent.  The exact calculation gives
\[
 (\ad(e+U))^{11}=0,\qquad
 \operatorname{rank}(\ad(e+U)|_{\fh})=51.
\]
The involution gives the same statements for $e+V$, and the centralizer
torus gives them for all nonzero parameters on each axis.  The orbit table
has exactly two nilpotent orbits of dimension 51, namely 21 and 22.
They are exchanged by $s$: the central pairing $B_0(Z,d)$ of their normal
triples has opposite nonzero values (with the table normalization,
$-6\ii$ and $6\ii$).  This pairing is invariant under $L$ and, by
Lemma~\ref{lem:triple-conjugacy}, under changing
the triple with fixed nilpositive element.  Thus the two axes meet the
two distinct orbits, and \eqref{evii:nil-contained-node} is equality as a
set near $e$.

\begin{proposition}[Vanishing along orbits $20$, $21$, and $22$]\label{evii:three-removed}
Every $L$-invariant ordinary distribution on $\fp$ which is odd under
the exchange and supported on $\cN$ vanishes near orbits 20, 21, and 22.
\end{proposition}
\begin{proof}
The action map $L\times(e+\fp^f)\to\fp$ is a submersion near $(1,e)$:
$\fp=[\fh,e]\oplus\fp^f$ by the triple decomposition.  Use the pullback convention of Lemma~\ref{lem:slice-density},
viewing the distribution first as a generalized function.  Left $L$-invariance makes the
pullback a constant generalized function in the group variable times a
slice distribution.  Right compensation by the centralizer torus makes
the slice distribution torus invariant.  Conjugation by $s$ makes it
odd under the exchange in \eqref{evii:slice}.

The group $L$ is unimodular, its finite order automorphism $s$ preserves
Haar density, and the torus and $s$ preserve the absolute real Lebesgue density
on the slice: the two torus weights are reciprocal, the other two are trivial,
and the exchange has complex determinant $-1$.  Lemma~\ref{lem:slice-density} therefore gives the same
equivariance for generalized functions and ordinary distributions.

By \eqref{evii:nil-contained-node} the slice distribution satisfies
Lemma~\ref{evii:node} with $W=\C w_4\oplus\C w_8$, and is zero.
The action submersion has open image, so the original distribution
vanishes on a neighborhood of $e$.  That neighborhood meets each
punctured axis and hence each of orbits 21,22.  Invariance makes its
support $L$-stable, so neither orbit can meet its support.  The same
argument applies to all points of orbit 20.
\end{proof}

\subsection{The character on orbits $17$ and $18$}
Suppose that $\mathcal M=\cC^{s=-1}$ is nonzero. Since $s$ preserves
the Killing form, this space is stable under Fourier transformation and
multiplication by the Killing quadratic form and its conjugate.
Indeed, Fourier transformation interchanges quadratic multiplication
with a constant coefficient differential operator, which preserves
support in $\cN$. Lemma~\ref{lem:four-character} selects a
nonzero homogeneous $T$. By Lemma~\ref{lem:normal-degrees}, the origin
and all orbits of negative defect cannot
be maximal in its support.  Lemma~\ref{evii:wavefront} excludes orbit 15, and
Proposition~\ref{evii:three-removed} excludes 20,21,22 altogether.
Every maximal support orbit is therefore 17 or 18.

Write $\eta_{a,b}(t)=\absC{t}^{27}t^a\bar t^b$.
On these two orbits the elementary normal degrees are 4 and 8.  Each of
$\eta_{0,0},\eta_{1,0},\eta_{0,1}$ requires degree 6 on at least one
side, which is impossible.  Thus
\begin{equation}\label{evii:remaining}
 \rho(t)T=|t|_{\C}^{28}T,\qquad
 \supp T,\supp\F T\subset\cN,\qquad s_*T=-T,
 \qquad \supp T\subset Z_{17,18}:=\overline{O_{17}}\cup\overline{O_{18}}.
\end{equation}
The next slice imposes a residue constraint and a stabilizer equation
on this character.

\subsection{The boundary slice at orbit 15}
To determine whether the distribution in \eqref{evii:remaining} can
extend across the common boundary of orbits $17$ and $18$, we construct
a slice at orbit $15$. Its intersection with $Z_{17,18}$ will again be
a node, with additional normal directions constrained by homogeneity.
Use the same root basis as above.  For orbit 15 take
\begin{align*}
 H_0&=4H_1+6H_2+8H_3+12H_4+10H_5+8H_6+4H_7,\\
 E&=\sqrt2(X_6+X_{34}+X_{40}+X_{56}),
\end{align*}
and obtain $(e_0,d_0,f_0)$ by the same Cayley formula.  Put
\[
 z_{15}=S_7-S_{20}-S_{21},\qquad
 s_{15}(X_\alpha)=(-1)^{a_4+a_7}X_\alpha,
 \qquad s_{15}(H_j)=H_j.
\]
Direct substitution shows that $s_{15}$ commutes with $\theta$, fixes
$(e_0,d_0,f_0)$, and reverses both $Z$ and $z_{15}$.  It differs from the
exchange used at orbit 20 by an element of $L$: their product centralizes
$Z$, whose centralizer in the adjoint group is the connected Levi $L$.
Consequently $s_{15*}T=-T$ in \eqref{evii:remaining}.

Write $C=L^{e_0,d_0,f_0,\circ}$ and $N=\fp^{f_0}$.  The kernel and bracket calculation of Appendix~\ref{app:E7} gives:
\begin{enumerate}
\item $\dim C=15$, $\dim[C,C]=14$, and
$\operatorname{Lie}C=[\operatorname{Lie}C,\operatorname{Lie}C]\oplus\C z_{15}$.
\item The $d_0$-weights on $N$ are $0,-2,-4$, of multiplicities $2,8,2$.
On the first and last spaces $\ad z_{15}$ has eigenvalues $2\ii,-2\ii$;
on the middle space it is zero.
\item Each of the four one dimensional nonzero central weight spaces is
$C$-stable.  In the middle space the $C$-fixed subspace is exactly $\C f_0$.
\end{enumerate}
The centralizer of a triple is reductive; its $\theta$-fixed identity
component $C$ is reductive as well.  The preceding facts imply
\[
 N^{[C,C]}=\C X\oplus\C Y\oplus\C f_0\oplus\C U\oplus\C V.
\]
Here $[d_0,X]=0$, $[d_0,U]=-4U$, and
$[z_{15},X]=2\ii X$, $[z_{15},U]=2\ii U$.  Normalize $X,U$ by requiring
their $H_7$ coefficients to be $\ii,2\ii$, respectively, and put
$Y=s_{15}X$, $V=s_{15}U$.  The remaining seven dimensional space $W$
is the sum of the nontrivial $[C,C]$-isotypic spaces.  It is preserved by
$s_{15}$ and has no $[C,C]$-fixed vectors.

Thus the slice is
\begin{equation}\label{evii:slice15}
 e_0+xX+yY+z f_0+w+uU+vV,\qquad w\in W.
\end{equation}
The centralizer torus has effective action
\[
 a(x,y,z,w,u,v)=(ax,a^{-1}y,z,w,au,a^{-1}v).
\]
The exchange swaps $x,y$ and $u,v$, fixes $z$, and preserves $W$.
For this triple, the Kazhdan dilation is
$\kappa_t(Y)=t^2\Ad(D_t^{-1})Y$, where $D_t$ is obtained from
$d_0$. It fixes $e_0$ and multiplies a coordinate of degree $r$
by $t^r$. Its positive degrees are
\begin{equation}\label{evii:degrees15}
 \deg(x,y,z,w,u,v)=(2,2,4,4,6,6).
\end{equation}

\begin{lemma}[The slice intersection]\label{evii:support15}
The set theoretic intersection of $Z_{17,18}$ with \eqref{evii:slice15} is
\[
 \{xy=0,\ z=w=u=v=0\}.
\]
The two punctured axes meet the two distinct orbits 17 and 18.
\end{lemma}
\begin{proof}
We first show that the intersection consists of curves, then use
the two commuting tori to force each curve into a coordinate axis.
Finally we identify which axes occur.

\emph{Dimension of the intersection.}
The action map $L\times(e_0+N)\to\fp$ is a submersion everywhere.
It is a submersion at $(1,e_0)$ by the triple decomposition. The
rank failure locus on $e_0+N$ is closed and stable under
$\kappa_t$. If it contained a point, contraction as $t\to0$ would
put $e_0$ in it, a contradiction. Equivariance then gives the
assertion at every point of $L\times(e_0+N)$.

The two orbit closures are equidimensional of dimension $43$,
whereas $\dim Le_0=42$. The action submersion has relative
dimension $\dim L-42$, and its inverse image of $Z_{17,18}$ is
\[
 L\times\bigl(Z_{17,18}\cap(e_0+N)\bigr).
\]
Thus every irreducible component of the slice intersection has
dimension one.

\emph{Reduction to coordinate axes.}
Give such a component $Y$ its reduced structure. The connected
groups $C$ and the Kazhdan torus preserve $Y$, since they preserve
the intersection and cannot permute its finitely many irreducible
components. Translating $e_0$ to the origin, the positive Kazhdan
weights make $\C[Y]$ a positively graded algebra.
Every nonzero point of $Y$ has a one dimensional Kazhdan orbit,
because its stabilizer is finite. This orbit is dense in $Y$.
Evaluation at a point of this orbit is injective on each
homogeneous component of $\C[Y]$: a homogeneous function
vanishing there vanishes on the orbit, hence on $Y$.
Each homogeneous component consequently has dimension at most one.

The action of $[C,C]$ commutes with Kazhdan dilation and preserves
each homogeneous component. On every nonzero such component it
therefore acts through an algebraic character. Since $[C,C]$ is
connected and semisimple, these characters are trivial.
It follows that $[C,C]$ acts trivially on $\C[Y]$, hence on $Y$.
Therefore every component lies in $e_0+N^{[C,C]}$.

On this five dimensional fixed space the pairs of Kazhdan and central
weights are
\[
 (2,1),\quad(2,-1),\quad(4,0),\quad(6,1),\quad(6,-1).
\]
No two are collinear.  A point with two nonzero coordinates would have
a two dimensional orbit under the product of these tori.  It cannot lie
in the one dimensional intersection.  Thus only coordinate axes remain.
\emph{Identification of the axes.}
The $f_0$-axis is not nilpotent away from zero, since
$B(e_0+tf_0,e_0+tf_0)=2tB(e_0,f_0)\ne0$.
The exact ranks at $e_0+U$ and $e_0+V$ are 51.  The central torus makes
this rank constant on each punctured axis, so neither axis can belong to
a union of closures of orbits of dimension 43.

For $e_x=e_0+X$, the element $d_x=d_0-\ii z_{15}$ completes with an explicitly
solved $f_x\in\fp$ to a normal triple.  The exact spectra are
\[
\begin{array}{c|c}
\fp&6^{(1)},\ 2^{(26)},\ (-2)^{(26)},\ (-6)^{(1)}\\
\fh&4^{(16)},\ 0^{(47)},\ (-4)^{(16)}.
\end{array}
\]
Here an exponent in parentheses denotes an eigenvalue multiplicity;
for example, $2^{(26)}$ means eigenvalue $2$ with multiplicity $26$.
The orbit of $e_x$ has dimension $43$, and its highest multiplicities identify it
as 17 or 18 in the table.  Moreover $B_0(Z,d_x)=2\ii$, while the exchanged
triple has pairing $-2\ii$.  By Lemma~\ref{lem:triple-conjugacy}, this pairing is independent
of the triple and invariant under $L$. It distinguishes the two
orbits.  The central torus covers both punctured axes, proving the
claimed equality of sets.
\end{proof}

\subsection{Homogeneous normal coefficients}
We now determine the distributions allowed on the two punctured branches
by the character in \eqref{evii:remaining}. The conversion between
ordinary distributions and generalized functions determines their precise
slice homogeneity; the centralizer action then restricts the normal
coefficients to a $3\times3$ matrix.

Pull back $T$ as a generalized function under the action submersion
at $e_0$. Left $L$-invariance makes the pullback constant in the
group variable. For $c\in C$, the transformation
\[
 (h,e_0+n)\longmapsto
 (hc^{-1},e_0+\Ad(c)n)
\]
leaves the action map unchanged. Thus the resulting generalized
function on the slice is $C$-invariant.
Conjugation by $s_{15}$ similarly makes it
$s_{15}$-odd.

These equivariances also hold after using constant Lebesgue
density on $N$ to obtain an ordinary slice distribution $S$.
Indeed, the determinant character on $[C,C]$ is trivial.
On the connected centre of $C$, the weights displayed above
occur in opposite pairs, so the determinant is again one.
Since $C=Z(C)^\circ[C,C]$, every element of $C$ preserves
the slice density. The involution $s_{15}$ preserves it as well.
Consequently $S$ is $C$-invariant and $s_{15}$-odd,
with support in the node of Lemma~\ref{evii:support15}.

To compute the slice homogeneity, let
$A_t=t^{-2}\Ad(D_t)$, with $D_t$ from $d_0$; on the slice it is inverse
Kazhdan dilation.  On ambient ordinary distributions its character is
$|t|_{\C}^{56}$ by \eqref{evii:remaining}.  Its absolute ambient determinant
is $|t|_{\C}^{-108}$.  Its absolute slice determinant, using
\eqref{evii:degrees15}, is $|t|_{\C}^{-48}$, since
$2+2+8\cdot4+6+6=48$.  Lemma~\ref{lem:slice-density} gives
\begin{equation}\label{evii:slice15-homog}
 (A_t)_*S=|t|_{\C}^{56-108+48}S=|t|_{\C}^{-4}S.
\end{equation}
The equivariance required in that lemma is given by
$(h,Y)\mapsto(D_thD_t^{-1},A_tY)$. The constant generalized
function on the group factor is fixed by this automorphism.

Expand $S$ on the punctured $x$-axis in all normal derivatives.  This is
a finite expansion on a smaller coordinate neighborhood because $S$ is a
distribution of finite order there.  Let $p,a,b$ be its holomorphic
derivative counts in $y,u,v$, and let $k$ be the total holomorphic count
in $z,W$.  Torus invariance determines the coefficient, up to a constant,
as $x^{a-b-p}d^\times x$, with the analogous antiholomorphic factor.
Inverse Kazhdan dilation gives holomorphic degree
\[
 2p+6a+6b+4k-2(a-b-p)=4(p+a+2b+k).
\]
Equation~\eqref{evii:slice15-homog} therefore says
\begin{equation}\label{evii:branch-count}
 p+a+2b+k=1,
\end{equation}
and gives the same equation for the antiholomorphic counts. All
counts are nonnegative integers. Thus $b=0$, and exactly one of
$p,a,k$ is one. Before imposing $[C,C]$-invariance, the possible
holomorphic directions are therefore $y$, $u$, or one direction in
$\C z\oplus W$. The antiholomorphic choices are independent.
Every surviving branch term has one derivative of each kind.

The group $[C,C]$ fixes the $x$-axis pointwise. The infinitesimal
equations for $q$ and $\ii q$, with $q\in\operatorname{Lie}[C,C]$,
separate the holomorphic and antiholomorphic tensor factors. Since
$W^{[C,C]}=0$, neither factor of a first order invariant normal
coefficient contains $W$.

Write the three holomorphic possibilities as
\begin{equation}\label{evii:ABC}
 A_x=x^{-1}\partial_y,\qquad B_x=\partial_z,
 \qquad C_x=x\partial_u.
\end{equation}
Derivatives here act on the normal delta distribution; the coefficient
on the axis is $d^\times x$.  Every branch distribution is consequently described by a unique
matrix of complex constants $M=(m_{ij})_{i,j=A,B,C}$ in
$\sum m_{ij}\,i_x\overline{j_x}\,d^\times x$. The row index
records the holomorphic choice $A_x,B_x,C_x$, and the column index
records the antiholomorphic choice. This is the nine dimensional
coefficient space used in the rest of the proof.
The distribution on the punctured $y$-axis is the negative of its
$s_{15}$-translate.

\subsection{Boundary residues and the stabilizer equation}
The branch matrix must satisfy two conditions. Extension across the
origin must preserve the central torus action, and the leading normal
symbol must be fixed by the stabilizer of a point on the branch.
The extension condition forces four entries to vanish; the stabilizer
equations then force the remaining five to vanish.

\begin{lemma}[The residue constraint]\label{evii:residue15}
If these two branches extend to a torus invariant distribution germ at
$e_0$, their coefficient matrix must satisfy
\[
 m_{AA}=m_{AB}=m_{BA}=m_{BB}=0.
\]
\end{lemma}
\begin{proof}
Use the finite part $F$ in \eqref{evii:finitepart}.  For the three
holomorphic coefficient types in \eqref{evii:ABC}, extension operators
on $F$ are respectively
\[
 -\partial_x\partial_y,\qquad \partial_z,
 \qquad x\partial_u,
\]
with normal deltas in all variables except $x$.  Compose with the
analogous antiholomorphic operators to obtain extensions of all nine
branch terms.  Away from zero this has exactly the required coefficients:
$-\partial_x(1/(x\bar x))=x^{-1}/(x\bar x)$.
Each combined operator has central weight zero, so its torus anomaly is
its application to $-2\log|a|\,\delta_0$.

For the $A,B$ types the nonzero holomorphic residues are, up to their
specified signs,
\[
 J_A=\partial_x\partial_y\delta_0,\qquad
 J_B=\partial_z\delta_0.
\]
The four resulting mixed point jets $J_i\overline{J_j}$, $i,j=A,B$,
mean the corresponding combined derivatives of a single delta.  They
are independent and fixed by both the central torus and $s_{15}$. A $C$ on
either side gives zero residue: $x\delta_0=0$,
$x\partial_{\bar x}\delta_0=0$, and their conjugates.  In the $CC$ case
$x\bar x\delta_0=0$ as well.

Take the extensions on the $y$-axis to be the negative exchanged
extensions.  Since $s_{15}as_{15}^{-1}=a^{-1}$, their anomalies have the
same signs as those on the $x$-axis for these four fixed jets.  The total
anomaly is twice the $x$-branch residue.  An arbitrary other extension
differs by a finite point jet.  Its torus coboundary has zero projection
to the trivial character part of point jets.  It cannot change any of
these four coefficients.  Their independence proves the assertion.
\end{proof}

We impose stabilizer invariance on the five remaining matrix entries.  At $e_x=e_0+X$
put
\begin{align*}
 q_*={}&-\sqrt2\ii X_{14}-\tfrac{4\sqrt2}{3}X_{28}
 -\tfrac13X_{41}-X_{50}\\
 &-\tfrac{8\sqrt2}{3}X_{-2}-\tfrac13X_{-19}
 +\sqrt2\ii X_{-26}+X_{-33}.
\end{align*}
Put
\[
 \mathcal T_x=[e_x,\fh],\qquad Q_x=\fp/\mathcal T_x,
\]
and write $[v]$ for the class of $v\in\fp$ in $Q_x$.
The bracket calculation gives $q_*\in\fh^{e_x}$, so $\ad(q_*)$
preserves $\mathcal T_x$ and induces an endomorphism of $Q_x$.
Define $W_*=[\, [q_*,Y]\,]\in Q_x$. Then
\begin{equation}\label{evii:qstar}
 W_*\ne0,\qquad
 q_*\cdot[Y]=W_*,\qquad
 q_*\cdot[f_0]=\tfrac14W_*,\qquad
 q_*\cdot[U]=\tfrac14W_*.
\end{equation}
The classes $[Y],[f_0],[U]$ are independent. These statements are
equivalently expressed by the following ranks:
\begin{equation}\label{eq:EVII-five-ranks}
\begin{array}{c|ccccc}
\text{space}
 &\mathcal T_x
 &\mathcal T_x+\C[q_*,Y]
 &\mathcal T_x+\C[q_*,4f_0-Y]
 &\mathcal T_x+\C[q_*,4U-Y]
 &\mathcal T_x+\langle Y,f_0,U\rangle\\\hline
\dim&43&44&43&43&46.
\end{array}
\end{equation}

The displayed ranks have two roles: they show that $[Y],[f_0],[U]$
are independent normal directions, and that $q_*$ sends them to
the same nonzero direction with coefficients $1,1/4,1/4$.
Thus a single stabilizer equation will relate the three rows of
$M$, and its conjugate will relate the three columns.

By Lemma~\ref{lem:EVII-unimodular}, the highest ordinary normal symbol
of $T$ is $L_{e_x}$-invariant without a density twist. The punctured
$x$-axis is the orbit intersection, so
$N\cap[e_x,\fh]=\C X$ and
$N/\C X\simeq\fp/[e_x,\fh]$.
Transverse pullback identifies the leading symbols through this
isomorphism. Thus $M$ represents the symbol at $x=1$, up to a common
nonzero density factor.

Applying $q_*$ and $\ii q_*$ again separates holomorphic and antiholomorphic
actions.  By \eqref{evii:qstar}, invariance implies
\begin{equation}\label{evii:matrix-relation}
 rM=0,\qquad Mr^{\mathsf t}=0,
 \qquad r=(1,\tfrac14,\tfrac14).
\end{equation}
By Lemma~\ref{evii:residue15} we can write
\[
 M=\begin{pmatrix}0&0&c\\0&0&f\\g&h&k\end{pmatrix}.
\]
The first equation in \eqref{evii:matrix-relation} gives $g=h=0$.
The second then gives $c=f=k=0$. Thus $M=0$.

Consequently $S$ vanishes on both punctured axes.  The action submersion
and $L$-invariance imply that $T$ vanishes on both orbits 17 and 18.
This contradicts their being its only possible maximal support orbits.
The nonzero $T$ in \eqref{evii:remaining} does not exist. This proves
Proposition~\ref{prop:evii}.

\section{Completion of the proof}\label{sec:completion}
We now combine the preceding results. Proposition~\ref{prop:nice}
gives $\cC=0$ for the nice blocks. For the remaining classical blocks,
the Gaussian criterion gives the stronger assertion that $K_Y$ fixes
every $L$-invariant tempered distribution. Seven exceptional blocks
follow from Proposition~\ref{prop:exceptional-aut}, including three
that are also nice; EIII and EVII
follow from the vanishing of the odd part of $\cC$. Thus every block
in the tables below satisfies \eqref{eq:F}.

The symmetric pairs of Lie algebras with simple ambient algebra are
listed below. The classification
and notation agree with \cite[Tables 2--4]{Rubio}; the nice families are
those of \cite[Definition 5.2.5]{Aiz}. The last column lists every method established here for the indicated
family or parameter range. ``Nice'' refers to the vanishing
$\cC=0$ in Proposition~\ref{prop:nice}; ``Gaussian'' refers to
the criterion of Theorem~\ref{thm:gaussian} as applied in
Proposition~\ref{prop:classical}; and ``automorphisms'' means that
$A=L$ or $K_Y=L$, as proved in Proposition~\ref{prop:classical}
or Proposition~\ref{prop:exceptional-aut}. The latter two methods
fix all $L$-invariant tempered distributions.

\begin{center}\small
\begin{tabular}{@{}lp{.49\textwidth}>{\raggedright\arraybackslash}p{.38\textwidth}@{}}
\toprule
Type & $(\fg,\fh)$ & Result\\
\midrule
AI & $(\mathfrak{sl}_n,\mathfrak{so}_n)$ & Nice; also Gaussian for $n=2$\\
AII & $(\mathfrak{sl}_{2n},\mathfrak{sp}_{2n})$
 & $n=2$: Gaussian; $n\geq3$: automorphisms\\
AIII & $(\mathfrak{sl}_{p+q},\mathfrak{s}(\mathfrak{gl}_p\oplus\mathfrak{gl}_q))$
 & $p=q$: nice; also Gaussian if $p=q\leq2$; $p\ne q$: Gaussian\\
BDI & $(\mathfrak{so}_{p+q},\mathfrak{so}_p\oplus\mathfrak{so}_q)$
 & $p=q$: nice; $|p-q|=1,2$: nice and Gaussian; $|p-q|>2$: Gaussian\\
CI & $(\mathfrak{sp}_{2n},\mathfrak{gl}_n)$ & Nice; also Gaussian for $n=1,2$\\
CII & $(\mathfrak{sp}_{2(p+q)},\mathfrak{sp}_{2p}\oplus\mathfrak{sp}_{2q})$
 & $p=q$: Gaussian; $p\ne q$: automorphisms\\
DIII & $(\mathfrak{so}_{2n},\mathfrak{gl}_n)$ & Gaussian\\
\bottomrule
\end{tabular}
\end{center}
Here $n,p,q$ range over positive integers for which the ambient Lie
algebra is simple and the involution is nontrivial. For BDI we have
extended Proposition~\ref{prop:classical} to all $p\ne q$ precisely
to make its overlap with the nice range explicit. Low dimensional
nonsimple cases are interpreted through their simple and central
decompositions. The rank one pair
$\mathfrak{sl}_2/\mathfrak{so}_2$, also denoted AIII with $p=q=1$
or CI with $n=1$, is BDI with dimensions $(2,1)$ and therefore has
both the nice and Gaussian proofs. Likewise, AIII with $p=q=2$
is BDI with dimensions $(4,2)$, and CI with $n=2$ is BDI with
dimensions $(3,2)$; these identifications give the additional Gaussian
entries in those rows. The identifications of AII with
$n=2$ and DIII with $n=4$ used in Proposition~\ref{prop:classical}
are included in their Gaussian entries.

\begin{center}\small
\begin{tabular}{@{}lp{.49\textwidth}>{\raggedright\arraybackslash}p{.38\textwidth}@{}}
\toprule
Type & $(\fg,\fh)$ & Result\\
\midrule
EI & $(\mathfrak e_6,\mathfrak{sp}_8)$ & Nice\\
EII & $(\mathfrak e_6,\mathfrak{sl}_6\oplus\mathfrak{sl}_2)$ & Nice\\
EIII & $(\mathfrak e_6,\mathfrak{so}_{10}\oplus\C)$
 & Homogeneity and normal symbols, Proposition~\ref{eiii:vanishing}\\
EIV & $(\mathfrak e_6,\mathfrak f_4)$ & Automorphisms\\
EV & $(\mathfrak e_7,\mathfrak{sl}_8)$ & Nice\\
EVI & $(\mathfrak e_7,\mathfrak{so}_{12}\oplus\mathfrak{sl}_2)$ & Automorphisms\\
EVII & $(\mathfrak e_7,\mathfrak e_6\oplus\C)$
 & Homogeneity and slices, Proposition~\ref{prop:evii}\\
EVIII & $(\mathfrak e_8,\mathfrak{so}_{16})$ & Nice and automorphisms\\
EIX & $(\mathfrak e_8,\mathfrak e_7\oplus\mathfrak{sl}_2)$ & Automorphisms\\
FI & $(\mathfrak f_4,\mathfrak{sp}_6\oplus\mathfrak{sl}_2)$ & Nice and automorphisms\\
FII & $(\mathfrak f_4,\mathfrak{so}_9)$ & Automorphisms\\
GI & $(\mathfrak g_2,\mathfrak{sl}_2\oplus\mathfrak{sl}_2)$ & Nice and automorphisms\\
\bottomrule
\end{tabular}
\end{center}
If the ambient algebra of an indecomposable semisimple symmetric pair
of Lie algebras is not simple, the pair is diagonal:
$(\mathfrak a\oplus\mathfrak a,\operatorname{diag}\mathfrak a)$,
with the involution exchanging the two simple ideals. These pairs are
also nice.

\begin{proof}[Proof of Theorem~\ref{thm:main}]
Every nontrivial indecomposable semisimple block is either a type in
the tables or a diagonal block. The preceding propositions establish
\eqref{eq:F} in every case, and Theorem~\ref{thm:assembly} gives
regularity of $(G,G^\theta,\theta)$.
\end{proof}

\begin{proof}[Proof of Corollary~\ref{cor:GK}]
Every descendant has the form $(G_x,(G_x)^\theta,\theta)$,
where $G_x=Z_G(x)$ is the centralizer
of a semisimple element arising from a closed double orbit
\cite[Proposition 7.2.1 and Definition 7.2.2]{AG}. Its identity component is
reductive, so Theorem~\ref{thm:main} applies to every descendant.
For a good pair, generalized Harish--Chandra descent
\cite[Theorem 7.4.5]{AG} gives \eqref{eq:GK}. A complex symmetric pair
with connected $G/H$ is good by \cite[Corollary 7.1.7]{AG}.

For any good pair, the Gelfand--Kazhdan criterion gives the product
bound for an irreducible Casselman--Wallach representation; see
\cite[Theorem 3.1]{Rubio}. If $\pi=U^\infty$ for an irreducible
unitary representation, its smooth contragredient is isomorphic to
$\overline\pi$. Complex conjugation identifies the two invariant
functional spaces antilinearly, so they have equal dimensions.
Their product is at most one, and the asserted unitary bound follows.
\end{proof}

\appendix
\section{Root coordinates and exceptional calculations}\label{app:E7}
This appendix specifies the root model and algebraic calculations used
in Sections~\ref{eiii:section} and~\ref{sec:EVII}. The $E_6$ subsystem
used for EIII consists of the roots with seventh coordinate zero.
All calculations take place over
\[
 K=\Q(\ii,\sqrt2,\sqrt3,\sqrt5).
\]
The root numbering is that of \cite{DjokovicE7}. The Cayley representatives
are specified below in the same root basis.
The files accompanying the article contain the full signed structure constants,
the involution, and executable exact arithmetic calculations.

\subsection{The root model}
Order the basis as
\[
 H_1,\ldots,H_7,\ X_1,\ldots,X_{63},\ X_{-1},\ldots,X_{-63}.
\]
The Cartan matrix has diagonal entries $2$ and off diagonal entries $-1$
for the edges
\[
 (1,3),\ (3,4),\ (4,5),\ (5,6),\ (6,7),\ (2,4).
\]
In the following table the seven digit word $a_1\cdots a_7$ denotes the coordinate
vector of $\alpha_r=\sum a_j\alpha_j$.
\begin{center}\small
\begin{tabular}{@{}rl@{\qquad}rl@{\qquad}rl@{}}
\toprule
$r$&$\alpha_r$&$r$&$\alpha_r$&$r$&$\alpha_r$\\\midrule
1 & $1000000$ & 22 & $0111100$ & 43 & $1112210$\\
2 & $0100000$ & 23 & $0101110$ & 44 & $1112111$\\
3 & $0010000$ & 24 & $0011110$ & 45 & $0112211$\\
4 & $0001000$ & 25 & $0001111$ & 46 & $1122210$\\
5 & $0000100$ & 26 & $1111100$ & 47 & $1122111$\\
6 & $0000010$ & 27 & $1011110$ & 48 & $1112211$\\
7 & $0000001$ & 28 & $0112100$ & 49 & $0112221$\\
8 & $1010000$ & 29 & $0111110$ & 50 & $1123210$\\
9 & $0101000$ & 30 & $0101111$ & 51 & $1122211$\\
10 & $0011000$ & 31 & $0011111$ & 52 & $1112221$\\
11 & $0001100$ & 32 & $1112100$ & 53 & $1223210$\\
12 & $0000110$ & 33 & $1111110$ & 54 & $1123211$\\
13 & $0000011$ & 34 & $1011111$ & 55 & $1122221$\\
14 & $1011000$ & 35 & $0112110$ & 56 & $1223211$\\
15 & $0111000$ & 36 & $0111111$ & 57 & $1123221$\\
16 & $0101100$ & 37 & $1122100$ & 58 & $1223221$\\
17 & $0011100$ & 38 & $1112110$ & 59 & $1123321$\\
18 & $0001110$ & 39 & $1111111$ & 60 & $1223321$\\
19 & $0000111$ & 40 & $0112210$ & 61 & $1224321$\\
20 & $1111000$ & 41 & $0112111$ & 62 & $1234321$\\
21 & $1011100$ & 42 & $1122110$ & 63 & $2234321$\\
\bottomrule
\end{tabular}
\end{center}
The brackets are
\begin{align*}
 [H_i,X_{\pm r}]&=\pm\Bigl(\sum_j A_{ij}a_j\Bigr)X_{\pm r},\\
 [X_r,X_{-r}]&=-\sum_j a_jH_j,\\
 [X_r,X_s]&=N_{r,s}X_t\quad\text{if }\alpha_r+\alpha_s=\alpha_t,
\end{align*}
and zero in the remaining root cases. The signed constants $N_{r,s}\in\{1,-1\}$
are the literal array \texttt{N\_DATA} in \texttt{verifiers/e7\_model.py}, in the
normalization of \cite[Appendix]{DjokovicE7}. The extension rules are
$N_{-r,-s}=N_{r,s}$ and antisymmetry. These constants satisfy the Jacobi identity and give the invariant form
\[
 B_0(H_i,H_j)=A_{ij},\qquad B_0(X_r,X_{-r})=-1.
\]

The EVII involution is determined on the simple root vectors by
\[
\begin{array}{c|rrrrrrr}
 j&1&2&3&4&5&6&7\\\hline
 \theta(X_j)&X_{-37}&X_2&X_3&X_4&X_5&X_{-40}&X_{-7}.
\end{array}
\]
For $X_{-j}$ the root index on the right is reversed, with the same sign;
the images of $H_j$ follow from the brackets. Extending by brackets determines the full matrix. It satisfies
$\theta^2=1$ and $\theta[x,y]=[\theta x,\theta y]$, with
\[
 \dim\ker(\theta-1)=79,\qquad\dim\ker(\theta+1)=54.
\]
The central vector $Z=S_7-S_{49}+S_{63}$ and the two toral exchanges in
Section~\ref{sec:EVII} are expressed in this same basis.

\subsection{Cayley representatives and weights}
For the following data, $H_0=\sum b_jH_j$ and $F$ is obtained from $E$ by
reversing its root indices. They satisfy
\[
 [H_0,E]=2E,\qquad[H_0,F]=-2F,\qquad[E,F]=-H_0.
\]
Set
\begin{equation}\label{eq:appendix-Cayley}
 e=\tfrac12(E-F+\ii H_0),\qquad d=\ii(E+F),\qquad
 f=\tfrac12(E-F-\ii H_0).
\end{equation}
Then $(e,d,f)$ is a normal triple. Each paired row lists
one representative; its second representative is obtained by the toral exchange
$s$ in \eqref{evii:actual-swap}. The paired representatives have the same
weight data. The values $B_0(Z,d)/\ii$, in the row order below, are
\[
 -2,\ 4,\ 0,\ -6,\ -2,\ 0,\ 2,\ 4,\ 0,\ 6,\ 2,\ 0,\ -6.
\]
They are nonzero precisely on paired rows and change sign under $s$.
Since this pairing is constant on an $L$-orbit and independent of the normal
triple containing its nilpotent element, the two representatives are distinct.
The orbit classification is \cite[Table 6]{DjokovicE7}.\footnote{In the root normalization used here, the representative in row
$6,7$ is $E=X_7-X_{49}+X_{63}$. Its Cayley element is
$d=\ii Z$, which acts trivially on $\fh$ and with weights $\pm2$ on
$\fp$. This replaces the printed minus sign on the last term of that row
in \cite[Table 6]{DjokovicE7}.}
\begin{center}\small
\begin{longtable}{@{}rll@{}}
\toprule
Orbits&$(b_1,\ldots,b_7)$&$E$\\\midrule\endfirsthead
\toprule Orbits&$(b_1,\ldots,b_7)$&$E$\\\midrule\endhead
1,2&$(2,2,3,4,3,2,1)$&$X_{63}$\\
3,4&$(2,3,4,6,5,4,2)$&$X_{49}-X_{63}$\\
5&$(2,3,4,6,5,4,2)$&$X_{49}+X_{63}$\\
6,7&$(2,3,4,6,5,4,3)$&$X_{7}-X_{49}+X_{63}$\\
8,9&$(2,3,4,6,5,4,3)$&$X_{7}+X_{49}+X_{63}$\\
10&$(4,4,6,8,6,4,2)$&$X_{1}+X_{37}+X_{55}+X_{61}$\\
11,12&$(4,5,7,10,8,6,3)$&$X_{27}+X_{39}+X_{49}+X_{53}-X_{54}$\\
13,14&$(6,7,10,14,11,8,4)$&$\sqrt{3}X_{1}+\sqrt{3}X_{37}+2X_{49}$\\
15&$(4,6,8,12,10,8,4)$&$\sqrt{2}X_{6}+\sqrt{2}X_{34}+\sqrt{2}X_{40}+\sqrt{2}X_{56}$\\
16,19&$(6,7,10,14,11,8,5)$&$\sqrt{3}X_{1}-X_{7}+\sqrt{3}X_{37}+2X_{49}$\\
17,18&$(6,7,10,14,11,8,5)$&$\sqrt{3}X_{1}+X_{7}+\sqrt{3}X_{37}+2X_{49}$\\
20&$(8,10,14,20,16,12,6)$&$2X_{1}+\sqrt{3}X_{6}+\sqrt{3}X_{19}+2X_{37}+\sqrt{3}X_{40}+\sqrt{3}X_{41}$\\
21,22&$(10,13,18,26,21,16,9)$&$\sqrt{5}X_{1}+2 \sqrt{2}X_{6}+3X_{7}+\sqrt{5}X_{37}+2 \sqrt{2}X_{40}$\\
\bottomrule
\end{longtable}
\end{center}
Write
\[
 p_j=\dim\ker(\theta+1,\ad d-j),\qquad
 h_j=\dim\ker(\theta-1,\ad d-j),
\]
and set $p(t)=\sum_{j\ge0}p_jt^j$, $h(t)=\sum_{j\ge0}h_jt^j$.
The negative weights have the same multiplicities. Exact row reduction gives:
\begin{center}\small
\begin{tabular}{@{}rllrr@{}}
\toprule
Orbits&$p(t)$&$h(t)$&$\tau(e)$&$\dim Le$\\\midrule
1,2&$20 + 16 t + t^{2}$&$47 + 16 t$&16&17\\
3,4&$2 + 16 t + 10 t^{2}$&$47 + 16 t$&16&26\\
5&$18 + 16 t + 2 t^{2}$&$31 + 16 t + 8 t^{2}$&32&26\\
6,7&$27 t^{2}$&$79$&0&27\\
8,9&$32 + 11 t^{2}$&$47 + 16 t^{2}$&32&27\\
10&$30 + 12 t^{2}$&$37 + 20 t^{2} + t^{4}$&44&33\\
11,12&$12 + 12 t + 5 t^{2} + 4 t^{3}$&$21 + 12 t + 12 t^{2} + 4 t^{3} + t^{4}$&48&38\\
13,14&$2 + 8 t + 9 t^{2} + 8 t^{3} + t^{6}$&$31 + 8 t + 8 t^{3} + 8 t^{4}$&52&42\\
15&$18 + 16 t^{2} + 2 t^{4}$&$31 + 16 t^{2} + 8 t^{4}$&60&42\\
16,19&$16 + 10 t^{2} + 8 t^{4} + t^{6}$&$31 + 16 t^{2} + 8 t^{4}$&44&43\\
17,18&$26 t^{2} + t^{6}$&$47 + 16 t^{4}$&60&43\\
20&$14 + 10 t^{2} + 8 t^{4} + 2 t^{6}$&$19 + 14 t^{2} + 9 t^{4} + 6 t^{6} + t^{8}$&84&50\\
21,22&$17 t^{2} + 9 t^{6} + t^{10}$&$31 + 16 t^{4} + 8 t^{8}$&84&51\\
\bottomrule
\end{tabular}
\end{center}
For a graded $\mathfrak{sl}_2$-module, raising by $e$ interchanges the two
parities. The highest weight multiplicities are therefore
\begin{equation}\label{eq:highest-differences}
 m_{j,\fp}=p_j-h_{j+2},\qquad m_{j,\fh}=h_j-p_{j+2}.
\end{equation}
Consequently
\[
 \tau(e)=\sum_{j\ge0}j\,m_{j,\fh},\qquad
 \dim Le=79-\sum_{j\ge0}m_{j,\fh}.
\]
These identities compute the last two columns and the normal multiplicities in
\eqref{eq:EVII-six}. They also show directly that the six orbits listed there
are the only nonzero orbits with $\tau(e)\ge54$.

\subsection{The two slice calculations}
The matrices used for the slice calculations are kernels of stacked linear
systems. For example,
\[
 \fp^f=\ker\!\begin{pmatrix}\theta+1\\\ad f\end{pmatrix},\qquad
 \fh^{e,d,f}=\ker\!\begin{pmatrix}\theta-1\\\ad e\\\ad d\\\ad f\end{pmatrix}.
\]
Restriction to a weight or central character space adds the corresponding
matrix $\ad d-j$ or $\ad z-\lambda$ to the stack.
Choosing bases of $\fp$ and $\fh$ reduces these systems to the corresponding
eigenspaces.

At orbit $20$ the normal kernel has dimensions $1,2,1$ at weights $-2,-4,-6$.
The two weight $-4$ lines have central eigenvalues $\pm6\ii$.
Normalize $U$ by its $H_7$ coefficient $-2\sqrt3$ and $W_8$ by its $H_1$
coefficient $4\ii$, and put $V=sU$. To specify the coordinates explicitly, let
$\mathsf A_r=X_r-X_{-r}$. The vectors used in the calculation are
\begin{align*}
U={}&-2 \sqrt{3}H_{7}+2 \ii\mathsf A_{6}+2 \sqrt{3} \ii\mathsf A_{7}-2\mathsf A_{12}+2\mathsf A_{13}\\
&-2 \ii\mathsf A_{19}-2\mathsf A_{35}+\ii\mathsf A_{38}+2 \ii\mathsf A_{40}-2 \ii\mathsf A_{41}\\
&-\mathsf A_{42}+\mathsf A_{43}-\mathsf A_{44}-2\mathsf A_{45}+\ii\mathsf A_{46}\\
&-\ii\mathsf A_{47}+\ii\mathsf A_{48}-\mathsf A_{51}\\
W_8={}&4 \ii H_{1}+2 \ii H_{3}-2 \ii H_{5}-4 \ii H_{6}-2 \ii H_{7}\\
&-3\mathsf A_{1}+\sqrt{3}\mathsf A_{6}+\sqrt{3}\mathsf A_{19}-3\mathsf A_{37}+\sqrt{3}\mathsf A_{40}\\
&+\sqrt{3}\mathsf A_{41}+\sqrt{3} \ii\mathsf A_{42}-\sqrt{3} \ii\mathsf A_{43}-\sqrt{3} \ii\mathsf A_{44}-\sqrt{3} \ii\mathsf A_{51}\\
&-\mathsf A_{55}-\mathsf A_{61}.
\end{align*}
The resulting direct evaluations are
\begin{equation}\label{eq:three-coefficients}
 q_2(e+f)=-160,\qquad q_4(e+W_8)=-12288,\qquad
 q_6(e+U+V)=-1179648.
\end{equation}
The prescribed torus and Kazhdan weights identify these evaluations with the
three coefficients in \eqref{evii:node-invariants}.
The same calculation gives $(\ad(e+U))^{11}=0$ and
$\rk(\ad(e+U)|_\fh)=51$.

At orbit $15$ the kernel dimensions are $2,8,2$ at weights $0,-2,-4$.
The triple centralizer has dimension $15$, its derived algebra has dimension
$14$. The normal space fixed by the derived group has dimensions
$2,1,2$ in these three weights. The middle fixed line is $\C f_0$.
The normalization by the $H_7$ coefficients specifies $X,U$ uniquely.
The normal triple at $e_x=e_0+X$ is obtained by solving
\[
 [e_x,f_x]=d_0-\ii z_{15},\qquad
 [d_0-\ii z_{15},f_x]=-2f_x,\qquad f_x\in\fp.
\]
Its spectra and the five ranks in \eqref{eq:EVII-five-ranks} are then exact matrix
calculations. Finally, the coefficient equations
\[
 m_{AA}=m_{AB}=m_{BA}=m_{BB}=0,\qquad
 rM=0,\quad Mr^{\mathsf t}=0,\qquad r=(1,1/4,1/4)
\]
form a $10\times9$ rational matrix of rank $9$.

\subsection{Computational supplement}
The accompanying directory \texttt{verifiers} contains the programs
listed below. All calculations use exact rational or algebraic
arithmetic. The file \texttt{e7\_model.py} contains the root constants,
the involution, and the representatives used in the tables.

\begin{center}\small
\begin{tabular}{@{}p{.37\textwidth}p{.58\textwidth}@{}}
\toprule
Program & Calculation\\
\midrule
\texttt{classical\_moment\_check.py}
 & Nine universal identities for the paired matrix model, with
   rectangular multiplication types checked.\\
\texttt{eiii\_trace\_check.py}
 & The twelve EIII characteristics, graded multiplicities,
   orbit dimensions, and the modular traces for orbits 10 and 11.\\
\texttt{evii\_symbol\_check.py}
 & The dilation character and final coefficient equations,
   conditional on the displayed E7 data.\\
\texttt{verify\_e7\_structure.py}
 & The root system, Jacobi identities, and invariance of $B_0$.\\
\texttt{e7\_model.py}
 & The involution, normal triples, graded spectra, central
   pairings, and the comparison of signs in the row for orbits 6 and 7.\\
\texttt{verify\_e7\_slices.py}
 & Both EVII slices, invariant coefficients, orbit dimensions,
   stabilizer ranks, and the final linear system.\\
\texttt{eiii\_quotient\_check.py}
 & The EIII normal triple, tangent rank, and quotient differential
   in \eqref{eiii:quotient-jacobian}.\\
\bottomrule
\end{tabular}
\end{center}

These programs certify the finite algebraic calculations used in the
proof. The invariant generation, support, pullback, and residue
arguments are given in the text. In particular, the first program
does not purport to verify the invariant theory of all classical groups.
From the source directory, run
\[
 \texttt{python verifiers/run\_all.py}.
\]
The programs require Python~3.10 or later and SymPy~1.14.0.
The accompanying \texttt{README.md} gives the commands and expected
results; \texttt{SHA256SUMS} identifies the source files.

\end{document}